\documentclass{amsart}
\usepackage{color}

\newtheorem{theorem}{Theorem}[section]
\newtheorem{lemma}{Lemma}[section]
\newtheorem{proposition}{Proposition}[section]

\theoremstyle{definition}
\newtheorem{definition}[theorem]{Definition}

\theoremstyle{remark}
\newtheorem{remark}[theorem]{Remark}

\numberwithin{equation}{section}

\begin{document}
\renewcommand{\theequation}{\arabic{section}.\arabic{equation}}

\title[Existence and regularity of ultradifferentiable solutions]{Existence and regularity of ultradifferentiable periodic solutions to certain vector fields.}

\author{Rafael B. Gonzalez}
\address{\small Departamento de Matem\'atica, Universidade Estadual de Maring\'a, Maring\'a, PR, 87020-900, Brazil}
\email{rbgonzalez@uem.br}



\begin{abstract}
We consider a class of first-order partial differential operators, acting on the space of ultradifferentiable periodic functions, and we describe their range by using the following conditions on the coefficients of the operators: the connectedness of certain sublevel sets, the dimension of the subspace generated by the imaginary part of the coefficients, and Diophantine conditions. In addition, we show that these properties are also linked to the regularity of the solutions. The results extend previous ones in Gevrey classes.
\end{abstract}
\maketitle




\section{Introduction}

In \cite{BDG1}, it was studied the existence of solutions in $\mathcal{C}^{\infty}(\mathbb{T}^{N+1})$ to first-order partial differential equations given by $Lu=f,$ in which $f\in\mathcal{C}^{\infty}(\mathbb{T}^{N+1})$ and $L$ is a vector field on $\mathbb{T}^{N+1}$ of the type 

\begin{equation}\label{L_introduction}
L=\frac{\partial}{\partial t}+\sum_{j=1}^{N}(a_j+ib_j)(t)\frac{\partial}{\partial x_j},
\end{equation}
where the coefficients $a_j$ and $b_j$ are real-valued smooth functions defined on
$\mathbb{T}^1$, for $j=1,\ldots,N,$ and the coordinates in $\mathbb{T}^{1}\times\mathbb{T}^N$ are denoted by $(t,x)=(t,x_1,\ldots,x_N).$

To be more precise, \cite{BDG1} presents a characterization to the closedness of the range property for the operator $L:\mathcal{C}^{\infty}(\mathbb{T}^{N+1})\rightarrow\mathcal{C}^{\infty}(\mathbb{T}^{N+1}).$ The results in \cite{BDG1} extends to a torus of arbitrary dimension, previous ones in dimensions two and three, which appear in \cite{Ho} and \cite{BDGK}, respectively.

When $L\mathcal{C}^{\infty}(\mathbb{T}^{N+1})$ is closed, we say that $L$ is globally solvable. This means that $L\mathcal{C}^{\infty}(\mathbb{T}^{N+1})=(\ker{}^tL)^{\circ},$ in which $^tL:\mathcal{D}'(\mathbb{T}^{N+1})\rightarrow\mathcal{D}'(\mathbb{T}^{N+1})$ is the transpose operator of $L,$ and \[(\ker{}^tL)^\circ=\{f\in\mathcal{C}^{\infty}(\mathbb{T}^{N+1});\,\,\langle\mu,f\rangle=0, \,\forall \,\,\mu\in\ker{}^tL\}.\]

In addition, concerning the regularity of the solutions, in \cite{BDG1} it was characterized the global hypoellipticity of $L.$ We recall that the vector field $L$ is said to be globally hypoelliptic if the conditions $\mu\in\mathcal{D}'(\mathbb{T}^{N+1})$ and $L\mu\in \mathcal{C}^{\infty}(\mathbb{T}^{N+1})$ imply that $\mu\in \mathcal{C}^{\infty}(\mathbb{T}^{N+1}).$

A next step was given in \cite{BDG3}, which presents characterizations to the global solvability and to the global hypoellipticity of $L$ in the Gevrey classes. It was proved that we may solve in Gevrey classes without loosing regularity.

Many authors have given attention to the problems of existence and(or) regularity of periodic solutions to partial differential equations. We would like to mention the references \cite{A,AJ,AD,BFP,B1,B,BCP,BD,BDGK,CC,GW,Ho}. As we may see in these references, the most recent articles have been concerned with periodic ultradifferentiable functions, all of them studying global hypoellipticity.     

Inspired by \cite{BDG1,BDG3} and by \cite{A,AD,BFP}, we now study the problems of the existence and regularity of solutions for the vector field $L$ on the space of periodic ultradifferentiable functions of Roumieu type, $\mathcal{E}_{\{\omega\}}(\mathbb{T}^{N+1}),$ in which $\omega:[0,\infty)\rightarrow\mathbb{R}$ is a weight function (see Definition \ref{defr1}). In this study, we assume that the coefficients $a_j+ib_j$ belong to $\mathcal{E}_{\{\omega\}}(\mathbb{T}^{1}).$ 

We stress that the main focus of this article is the existence of solutions on $\mathcal{E}_{\{\omega\}}(\mathbb{T}^{N+1}).$ As in the Gevrey classes, we will show that we may solve without loos of regularity, that is, if $f\in\mathcal{E}_{\{\omega\}}(\mathbb{T}^{N+1})$ and supposing that $f$ satisfies the compatibility conditions, then we may find a solution to $Lu=f$ in the same class $\mathcal{E}_{\{\omega\}}(\mathbb{T}^{N+1}).$

If $\omega(t)=t^{1/s},$ $s>1$ and $t\geq0,$ then the space $\mathcal{E}_{\{t^{1/s}\}}(\mathbb{T}^{N+1})$ is the space of the periodic Gevrey functions of order $s.$ It follows that the results obtained in this article extend those in \cite{BDG3}. 

Motivated by \cite{BDG3}, we say that $L$ is globally $\{\omega\}-$solvable if the range of the operator $L:\mathcal{E}_{\{\omega\}}(\mathbb{T}^{N+1})\rightarrow \mathcal{E}_{\{\omega\}}(\mathbb{T}^{N+1})$ is equal to the subspace $(\ker{}^tL)^\circ,$ in which
\[(\ker{}^tL)^\circ=\left\{f\in\mathcal{E}_{\{\omega\}}(\mathbb{T}^{N+1});\,\langle \mu,f\rangle=0, \,\forall \,\,\mu\in\mathcal{E}_{\{\omega\}}'(\mathbb{T}^{N+1}) \,\,\mathrm{such\,\, that }{}^tL\mu=0\right\};\] here $\mathcal{E}'_{\{\omega\}}(\mathbb{T}^{N+1})$ denotes the strong dual of $\mathcal{E}_{\{\omega\}}(\mathbb{T}^{N+1}).$ The space $\mathcal{E}_{\{\omega\}}(\mathbb{T}^{N+1})$ is endowed  with an inductive limit topology given by a sequence of Banach spaces (see Section \ref{sec3}). It follows that $L$ is globally $\{\omega\}-$solvable if and only if $L\mathcal{E}_{\{\omega\}}(\mathbb{T}^{N+1})$ is a closed subspace of $\mathcal{E}_{\{\omega\}}(\mathbb{T}^{N+1})$ (see Lemma 2.2 in \cite{Ar}).

Following \cite{BCP}, $L$ is said to be strongly $\{\omega\}-$solvable if $L\mathcal{E}_{\{\omega\}}(\mathbb{T}^{N+1})$ is a closed subspace of $\mathcal{E}_{\{\omega\}}(\mathbb{T}^{N+1})$ and $\ker{}^tL$ is a finite dimensional subspace of $\mathcal{E}'_{\{\omega\}}(\mathbb{T}^{N+1}).$

The approach used in this article is mainly inspired in \cite{AD,B1,BDG1,BDG3}. One of the main tools is the use of partial Fourier series. Similar to which happens in $\mathcal{C}^\infty$ and in Gevrey spaces, the functions in $\mathcal{E}_{\{\omega\}}(\mathbb{T}^{N+1})$ may be characterized by partial Fourier series. This characterization is proved in Section 2 of \cite{BFP}.

As in \cite{BDG3}, the use of Fa\`{a} Di Bruno's formula for a derivative of the composed of two functions helps us to control the decaying of the partial Fourier coefficients. 

Another tool which appears in \cite{BDG1,BDG3} is the use of cutoff functions. This use is inspired in a technique which appears in \cite{Ho,T1}, where the existence of solutions is linked to the connectetdness of certain sublevel and superlevel sets. 

This strategy is used to prove Theorem \ref{mtr2}, which yields a characterization to the global $\{\omega\}-$solvability of $L:\mathcal{E}_{\{\omega\}}(\mathbb{T}^{N+1})\rightarrow \mathcal{E}_{\{\omega\}}(\mathbb{T}^{N+1})$ on the spaces of non-quasianalytic ultradifferentiable functions of Roumieu type. 

The second main result is Theorem \ref{mtr1}, which characterizes the strong $\{\omega\}-$ solvability and the global $\{\omega\}-$hypoellipticity of $L:\mathcal{E}_{\{\omega\}}(\mathbb{T}^{N+1})\rightarrow \mathcal{E}_{\{\omega\}}(\mathbb{T}^{N+1})$ on both the spaces of quasianalytic and non-quasianalytic ultradifferentiable functions of Roumieu type. As in \cite{Ar,BFP}, we say that $L$ is globally $\{\omega\}-$hypoelliptic if the conditions $\mu\in\mathcal{E}'_{\{\omega\}}(\mathbb{T}^{N+1})$ and $L\mu\in\mathcal{E}_{\{\omega\}}(\mathbb{T}^{N+1})$ imply that $\mu\in\mathcal{E}_{\{\omega\}}(\mathbb{T}^{N+1}).$ To treat the quasianalytic case, we will use that a quasianalytic function may be determined by its value and the values of all its successive derivatives at a singular point. In particular, the zeros of a non-identically zero quasianalytic function are isolated. Hence, we may apply certain techniques which appear in \cite{AD,B}. A crucial tool in this approach is the use of a method of stationary phase. The results in the quasianalytic context obtained in this article extends Theorem 2.3 in \cite{B}. We stress that the global analytic-solvability (in the sense of closed range) of vector fields of tube type (as the ones in \eqref{L_introduction}) is not completely understood. Therefore, here we do not touch this question in the quasianalytic setup.   

This article is organized as follows: In Section \ref{sec2} we state the main results and in Section \ref{sec3} we present the main properties about the class of ultradifferentiable functions which will be used here. The proof of Theorem \ref{mtr2} is done in Section \ref{sec4} and in Section \ref{sec5}. We dedicate Section \ref{sec6} to prove Theorem \ref{mtr1}.

\section{Statement of the main results}\label{sec2}

As in \cite{BDGK,BDG1,BDG3,GW}, both the problems of existence and regularity of periodic solutions are linked with a type of Diophantine (or exponential Diophantine) condition. For a pair of vectors $(\alpha,\beta)\in\mathbb{R}^N\times\mathbb{R}^N$ we define the following two exponential Diophantine conditions:\vspace{5pt}

\noindent$(EDC)_{1}^{\{\omega\}}$ for any $\epsilon>0$ there exists a constant $C_\epsilon>0$ such that
\[
|\tau+\langle\xi,\alpha+i\beta\rangle|\geq C_\epsilon\exp\{-\epsilon\omega(|(\tau,\xi)|)\},
\]
for all $(\tau,\xi)=(\tau,\xi_1,\ldots,\xi_N)\in\mathbb{Z}^{N+1}\setminus\{0\}.$
\vspace{5pt}

\noindent$(EDC)_2^{\{\omega\}}$ for any $\epsilon>0$ there exists a constant $C_\epsilon>0$ such that
\[
|\tau+\langle\xi,\alpha+i\beta\rangle|\geq C_\epsilon\exp\{-\epsilon\omega(|(\tau,\xi)|)\},
\]
for all $(\tau,\xi)=(\tau,\xi_1,\ldots,\xi_N)\in\mathbb{Z}^{N+1}$ such that $\tau+\langle\xi,\alpha+i\beta\rangle\neq0.$
\vspace{5pt}

We will use the following auxiliary notation: define \[a_{j0}=\frac{1}{2\pi}\int_{0}^{2\pi}a_j(t)dt, \quad b_{j0}=\frac{1}{2\pi}\int_{0}^{2\pi}b_j(t)dt, \quad j=1,\ldots,N, \]\[\alpha_{0}=(a_{10},\ldots,a_{N0}), \quad \beta_{0}=(b_{10},\ldots,b_{N0}),\]
\[\alpha(t)=(a_1(t),\ldots,a_{N}(t)) \quad \textrm{and} \quad \beta(t)=(b_1(t),\ldots,b_{N}(t)), \quad t\in\mathbb{T}^1.\]

We also define the sublevel sets
\begin{equation}\label{eqr7}
\Omega_r^{\xi}\doteq\left\{t\in\mathbb{T}^1;\int_{0}^{t}\langle\xi,\beta(\tau)\rangle d\tau<r\right\}, \,\,r\in\mathbb{R} \,\, \textrm{and} \,\,\xi\in\mathbb{Z}^N.
\end{equation}

\begin{theorem}\label{mtr2}
Suppose that $\omega$ is non-quasianalytic and subadditive, and let $L$ be given by {\em(\ref{L_introduction})}.
\begin{itemize}
\item[(I)] Assume that $b_{j0}=0,$ for every $j.$
\begin{itemize}
\item[(I.1)] If $b_j\equiv0$ for every $j,$ then $L$ is globally $\{\omega\}-$solvable if and only if the pair
$(\alpha_0,0)$ satisfies $(EDC)_2^{\{\omega\}}$.

\item[(I.2)] If $b_j\not\equiv0$ for at least one $j,$ then
$L$ is globally $\{\omega\}-$solvable if and only if $(a_{10},\ldots,a_{N0})\in\mathbb{Z}^N$ and all the sublevel sets $\Omega_r^{\xi}$ are connected.
\end{itemize}
\smallskip
\item[(II)] If $b_{j0}\neq0$ for some $j,$ then $L$ is globally $\{\omega\}-$solvable if and only if
the following properties are satisfied:
\begin{itemize}
\item[(II.1)] $\dim \emph{\textrm{span}}\{b_1,\ldots,b_N\}=1;$

\item[(II.2)] the functions $b_j$ do not change sign;

\item[(II.3)] the pair $(\alpha_{0},\beta_{0})$ satisfies $(EDC)_{2}^{\{\omega\}}$.
\end{itemize}
\end{itemize}
\end{theorem}

\begin{remark} The assumption ``$\omega$ is non-quasianalytic'' is mainly used in the situation described in item (I.2), to prove that $L$ is not globally $\{\omega\}-$solvable in the presence of a disconnected sublevel set. We stress that the global solvability (closedness of the range) is not completely understood, even on the space of the real-analytic functions, which has better properties than the quasianalytic functions. 

Condition (II.1) means that the functions $b_j$ are real multiples of a smooth real-valued function $b$ defined on $\mathbb{T}^1.$  
\end{remark}

The various conditions appearing in items (I) and (II) in Theorem \ref{mtr2} present a well detailed way to check whether operator $L$ is globally $\{\omega\}-$solvable. As in \cite{BCG}, there exists a shorter and equivalent way to describe the global $\{\omega\}-$solvability.

\begin{theorem}\label{mtr2.1}
Suppose that $\omega$ is non-quasianalytic and subadditive. Operator $L$ given by {\em(\ref{L_introduction})} is globally $\{\omega\}-$solvable if and only if the following conditions hold:
\begin{itemize}
\item whenever $\langle\xi,\beta(t)\rangle$ changes sign, we must have $\langle\xi,\beta_0\rangle=0,$ $\langle\xi,\alpha_0\rangle\in\mathbb{Z}$ and $\Omega_r^{\xi}$ is connected, for all $r\in\mathbb{R}.$

\item the pair $(\alpha_{0},\beta_{0})$ satisfies $(EDC)_{2}^{\{\omega\}}$.
\end{itemize}
\end{theorem}

Excluding the connectedness of the sublevel sets $\Omega_{r}^{\xi},$ the other properties concerning the coefficients $a_j$ and $b_j,$ which appear in Theorem \ref{mtr2},  are also linked to hypoellipticity.

\begin{theorem}\label{mtr1}
Let $L$ be given by {\em(\ref{L_introduction})} and suppose that $\omega$ is a subadditive weight function. Under these assumptions, the following conditions are equivalent:
\begin{itemize}
\item $L$ is strongly  $\{\omega\}-$solvable.
\item $L$ is globally  $\{\omega\}-$hypoelliptic. 
\item The functions $a_j$ and $b_j$ satisfy the following properties:
\begin{itemize}
\item[(1)] the functions $b_j$ do not change sign;

\item[(2)] $\dim\emph{\textrm{span}}\{b_1,\ldots,b_N\}\leq1$;

\item[(3)] the pair $(\alpha_{0},\beta_{0})$ satisfies $(EDC)_1^{\{\omega\}}$.
\end{itemize}
\end{itemize}
\end{theorem}

As before, there exists a shorter way to state Theorem \ref{mtr1}.

\begin{theorem}
Let $L$ be given by {\em(\ref{L_introduction})} and suppose that $\omega$ is a subadditive weight function. Under these assumptions, the following conditions are equivalent:
\begin{itemize}
\item $L$ is strongly  $\{\omega\}-$solvable.
\item $L$ is globally  $\{\omega\}-$hypoelliptic. 
\item for each $\xi,$ the function $t\mapsto\langle\xi,\beta(t)\rangle$ does not change sign, and the pair $(\alpha_{0},\beta_{0})$ satisfies $(EDC)_{1}^{\{\omega\}}.$
\end{itemize}
\end{theorem}

Although the statement of the results in this article are similar to the ones in \cite{BDG1, BDG3} we stress that their proofs present significant novelties and improvements. For instance, we invite the reader to compare the proofs of propositions \ref{pror2} and \ref{pror4} to the proofs of propositions 4.1 and 4.3 (concerning item III.3)  in \cite{BDG3}.

\section{Preliminaries}\label{sec3}

In this section, we present the space of periodic ultradifferentiable functions of Roumieu type, as well as its main properties which we will use throughout this article.  
 
We start recalling the definition of weight function.

\begin{definition}\label{defr1}An increasing continuous function $\omega:[0,\infty)\rightarrow[0,\infty)$ is said to be a weight function if the following properties hold:
\begin{itemize}
\item there exists $K\geq0$ such that $\omega(2t)\leq K(\omega(t)+1)$ for all $t\geq0,$
\item $\omega(t)=O(t)$ as $t$ tends to $\infty,$
\item $\log(t)=o(\omega(t))$ as $t$ tends to $\infty,$
\item $\varphi(t)\doteq\omega(e^t),$ for $t\geq0$ is convex.
\end{itemize}
\end{definition} 
 
We say that a weight function $\omega$ is quasianalytic if \[\int_{1}^\infty\frac{\omega(t)}{t^2}dt=\infty.\] On the other hand, if the above integral is finite, then $\omega$ is said to be a non-quasianalytic weight function. 

Given two weight functions $\sigma$ and $\omega,$ the notation $\sigma\preceq\omega$ means $\omega(t)=O(\sigma(t)),$ as $t$ tends to $\infty.$
We say that $\sigma$ and $\omega$ are equivalent if $\sigma\preceq\omega$ and $\omega\preceq\sigma.$ In other words, two weight functions $\omega$ and $\sigma$ are equivalent if \[0<\liminf_{t\rightarrow\infty}\frac{\omega(t)}{\sigma(t)}\leq\limsup_{t\rightarrow\infty}\frac{\omega(t)}{\sigma(t)}<\infty.\]
 
A weight function $\omega$ is equivalent to a subadditive weight function if and only if 
\begin{equation}\label{swf} \exists \ D>0, \exists \ t_0>0 \ \textrm{such that } \ \omega(\lambda t) \leq \lambda D\omega(t), \ \textrm{for all} \ \lambda\geq1 \ \textrm{and} \  t\geq t_0. 
\end{equation}

The Young conjugate $\varphi^\ast$ associated to a weight function $\omega$ is defined by $ \varphi^\ast(t)=\sup_{s\geq0}\{st-\varphi(s)\}.$
 
There is no loss of generality in assuming that $\omega$ vanishes on $[0,1].$ In this case, the following properties holds:
\begin{itemize} 
\item $\varphi^\ast([0,\infty))\subset[0,\infty),$
\item $\varphi^\ast$ is convex, superadditive, and $\varphi^\ast(0)=0,$
\item $\varphi^\ast(t)/t$ is increasing and tends to $\infty$ as $t\rightarrow\infty,$
\item $\varphi^{\ast\ast}=\varphi.$
\end{itemize} 
 
Since $\varphi^\ast$ is convex and $\varphi^\ast(0)=0,$ it follows that 
\[\varphi^{\ast}(t)+\varphi^\ast(s)\leq \varphi^\ast(t+s)\leq 2^{-1}[ \varphi^\ast(2t)+\varphi^\ast(2s)],
\] for all $t,s>0.$
 
In addition, when the weight function is subadditive, for each $\lambda>0$ and $j,k\in\mathbb{N}$ we have (see \cite{FGJ} Lemma 3.3)
\begin{equation}\label{eqr2}
e^{\lambda^{-1}\varphi^{\ast}(j\lambda)}e^{\lambda^{-1}\varphi^{\ast}(k\lambda)}\leq \frac{j!k!}{(j+k)!}e^{\lambda^{-1}\varphi^{\ast}((j+k)\lambda)}.
\end{equation}  
 
The space of periodic ultradifferentiable functions of Roumieu type is given by 
\[\mathcal{E}_{\{\omega\}}(\mathbb{T}^n)=\{f\in\mathcal{C}^\infty(\mathbb{T}^n); \ \|f\|_{\lambda}<\infty, \ \textrm{for some} \ \lambda>0\},\] in which
\[\|f\|_{\lambda}=\sup_{x\in\mathbb{T}^n}\sup_{\alpha\in\mathbb{Z}_{+}^{n}}|\partial^\alpha f(x)|\exp\left(-\lambda\varphi^\ast\left(\frac{|\alpha|}{\lambda}\right)\right).\]
 
When $\sigma\preceq\omega$ we have $\mathcal{E}_{\{\sigma\}}(\mathbb{T}^n)\subseteq\mathcal{E}_{\{\omega\}}(\mathbb{T}^n).$ 
 
In our approach to study the existence of solutions of \eqref{L_introduction}, the use of partial Fourier series leads us to solve certain linear ordinary differential equations. With this in mind, it is crucial that each space $\mathcal{E}_{\{\omega\}}(\mathbb{T}^n)$ contains $\mathcal{E}_{\{t\}}(\mathbb{T}^n)\doteq\mathcal{A}(\mathbb{T}^n),$ the real analytic functions. Since $\omega(t)=O(t),$ as $t$ tends to $\infty,$ we have $\mathcal{A}(\mathbb{T}^n)\subseteq\mathcal{E}_{\{\omega\}}(\mathbb{T}^n).$  
 
The usual topology on $\mathcal{E}_{\{\omega\}}(\mathbb{T}^n)$ is the inductive limit topology \[\mathcal{E}_{\{\omega\}}(\mathbb{T}^n)=\mbox{ind}\lim_{k\in\mathbb{N}}\mathcal{E}_{\{\omega\},k}(\mathbb{T}^n),\] with 
\[\mathcal{E}_{\{\omega\},k}(\mathbb{T}^n)=\{f\in\mathcal{C}^\infty(\mathbb{T}^n); \ \|f\|_{1/k}<\infty\}.\] 
 
We notice that $(\mathcal{E}_{\{\omega\},k}(\mathbb{T}^n),\|\cdot\|_{1/k})$ is a Banach space. In addition, for $k_1<k_2$ the inclusion $\mathcal{E}_{\{\omega\},k_1}(\mathbb{T}^n)\hookrightarrow\mathcal{E}_{\{\omega\},k_2}(\mathbb{T}^n)$ is compact. It follows that $\mathcal{E}_{\{\omega\}}(\mathbb{T}^n)$ is a (Hausdorff) locally convex TVS and $F\subset \mathcal{E}_{\{\omega\}}(\mathbb{T}^n)$ is closed if and only if $F\cap\mathcal{E}_{\{\omega\},k}(\mathbb{T}^n)$ is a closed set in $\mathcal{E}_{\{\omega\},k}(\mathbb{T}^n)$, for all $k\in\mathbb{Z}_+.$
 
\bigskip 

We recall that in the category of smooth functions the global hypoellipticity of $L$ implies that $L$ is strongly solvable (since $-L={}^tL$). This result also holds true in Gevrey classes (see \cite{AZ}). In \cite{FPV}, by following the approach of \cite{Ar}, this result was proved to the space of ultradifferentiable functions in the sense of Denjoy-Carleman. We also use the approach of \cite{Ar} to prove this result when $L$ acts on $\mathcal{E}_{\{\omega\}}(\mathbb{T}^{N+1}).$

\begin{lemma}\label{lemfinal} If ${}^tL:\mathcal{E}'_{\{\omega\}}(\mathbb{T}^{N+1})\rightarrow\mathcal{E}'_{\{\omega\}}(\mathbb{T}^{N+1})$ is globally $\{\omega\}-$hypoelliptic, then $L:\mathcal{E}_{\{\omega\}}(\mathbb{T}^{N+1})\rightarrow\mathcal{E}_{\{\omega\}}(\mathbb{T}^{N+1})$ is strongly $\{\omega\}-$ solvable.
\end{lemma}

\begin{proof} By Lemma 1.7 and Remark 1.8 in \cite{BMT} we may pick a subadditive weight function $\sigma$ such that $\sigma(t)=o(\omega(t)).$ Then for each $k$ we have the continuous inclusion 
$\mathcal{E}_{\{\omega\},k}(\mathbb{T}^{N+1})\hookrightarrow\mathcal{E}_{\{\sigma\},1}(\mathbb{T}^{N+1}).$ In particular, the inclusion $\mathcal{E}_{\{\omega\}}(\mathbb{T}^{N+1})\hookrightarrow\mathcal{E}_{\{\sigma\},1}(\mathbb{T}^{N+1})$ is continuous. 

It follows that \[\mathcal{E}_{\{\omega\}}(\mathbb{T}^{N+1})\hookrightarrow \mathcal{E}_{\{\sigma\}}(\mathbb{T}^{N+1})\hookrightarrow\mathcal{E}'_{\{\sigma\}}(\mathbb{T}^{N+1})\hookrightarrow\mathcal{E}'_{\{\omega\}}(\mathbb{T}^{N+1}).\] 

Since $L:\mathcal{E}_{\{\sigma\}}(\mathbb{T}^{N+1})\rightarrow\mathcal{E}_{\{\sigma\}}(\mathbb{T}^{N+1})$ is continuous, its graph $\tilde{\Gamma}=\{(u,Lu); u\in \mathcal{E}_{\{\sigma\}}(\mathbb{T}^{N+1})\}$ is a closed subspace of $\mathcal{E}_{\{\sigma\}}(\mathbb{T}^{N+1})\times \mathcal{E}_{\{\sigma\}}(\mathbb{T}^{N+1}).$  

Denoting the graph of $L:\mathcal{E}_{\{\omega\}}(\mathbb{T}^{N+1})\rightarrow\mathcal{E}_{\{\omega\}}(\mathbb{T}^{N+1})$ by $\Gamma$ and using the assumption that $L$ is globally $\{\omega\}-$hypoelliptic, we see that \[\Gamma=\tilde{\Gamma}\cap\left(\mathcal{E}_{\{\sigma\}}(\mathbb{T}^{N+1})\times \mathcal{E}_{\{\omega\}}(\mathbb{T}^{N+1})\right).\]  

Therefore, $\Gamma$ is a closed subspace of $\mathcal{E}_{\{\sigma\}}(\mathbb{T}^{N+1})\times\mathcal{E}_{\{\omega\}}(\mathbb{T}^{N+1}).$

Finally, by Theorem 2.5 in \cite{Ar} (with $E_0=\mathcal{E}_{\{\sigma\},1}(\mathbb{T}^{N+1}))$ it follows that the range $L\mathcal{E}_{\{\omega\}}(\mathbb{T}^{N+1})$ is closed in $\mathcal{E}_{\{\omega\}}(\mathbb{T}^{N+1})$ and $\ker L$ is finitely generated. Since ${}^tL=-L$ and ${}^tL$ is globally $\{\omega\}-$hypoelliptic, it follows that $\ker{}^tL$ is finitely generated. Therefore, $L$ is strongly $\{\omega\}-$ solvable.
\end{proof}

\bigskip
 
The following version of Lemma 3.10 in \cite{P1} and Lemma 6.1.2 in \cite{H}  will be used to link the connectedness of  sublevel sets with the global $\{\omega\}-$solvability. 

\begin{lemma}\label{lr1} If a linear partial differential operator $L,$ with coefficients in $\mathcal{E}_{\{\omega\}}(\mathbb{T}^{N+1}),$ is globally $\{\omega\}-$solvable, then for every $h>0$ and $k>0$ there exists $C>0$ such that 
\[\left|\int_{\mathbb{T}^{N+1}}fv\right|\leq C\|f\|_{1/h}\|{}^tLv\|_{1/k},\] for all $f\in(\ker{}^tL)^\circ\cap\mathcal{E}_{\{\omega\},h}(\mathbb{T}^{N+1})$ and $v\in\mathcal{E}_{\{\omega\}}(\mathbb{T}^{N+1})$ such that ${}^tLv\in\mathcal{E}_{\{\omega\},k}(\mathbb{T}^{N+1}).$
\end{lemma}

\begin{proof}Define \[V=\{v\in\mathcal{E}_{\{\omega\}}(\mathbb{T}^{N+1});^tLv\in\mathcal{E}_{\{\omega\},k}(\mathbb{T}^{N+1})\}\] and \[\mathcal{N}=\frac{V}{V\cap\ker{}^tL}.\] On $\mathcal{N}$ we consider the norm \[\|\bar{v}\|_\mathcal{N}=\|{}^tLv\|_{1/k}.\]

We also consider the Banach space $\mathcal{B}=((\ker{}^tL)^\circ\cap\mathcal{E}_{\{\omega\},h}(\mathbb{T}^{N+1}),\|\cdot\|_{1/h}).$ 

The bilinear form $B:\mathcal{B}\times\mathcal{N}\rightarrow\mathbb{C},$ given by $B(f,\bar{v})=\langle f,v\rangle,$ is well-defined. 

Notice that, for $(f,\bar{v})\in\mathcal{B}\times\mathcal{N}$ we have 
\[|B(f,\bar{v})|\leq (2\pi)^{N+1}\|f\|_\infty\|v\|_{\infty}\leq(2\pi)^{N+1}\|f\|_{1/h}\|v\|_{\infty},\] and picking $u$ such that $Lu=f,$ we have
\[|B(f,\bar{v})|=|\langle u,{}^tLv\rangle|\leq (2\pi)^{N+1}\|u\|_\infty\|{}^tLv\|_{1/k}.\] 

Above estimates imply that $B$ is separately continuous. It follows that $B$ is continuous. Hence, there exists $C>0$ such that \[|B(u,\bar{v})|\leq C\|f\|_{1/h}\|\bar{v}\|_\mathcal{N}=C\|f\|_{1/h}\|{}^tLv\|_{1/k}.\]

\end{proof}

Denoting by (t,x) the coordinates in $\mathbb{T}^{n+m},$ it follows that $\psi\in\mathcal{E}_{\{\omega\}}(\mathbb{T}^{n+m})$ if and only if there exist constants $C,\lambda,\epsilon>0$ such that 
\begin{equation}\label{eqr10}|\partial_t^{\alpha}\hat{\psi}(t,\xi)|\leq C\exp\left\{\lambda\varphi^\ast\left(|\alpha|\lambda^{-1}\right)\right\}\exp\{-\epsilon\omega(|\xi|)\},\end{equation} for all $t\in\mathbb{T}^n,$ for all $\alpha\in\mathbb{Z}_+^n$ and for all $\xi\in\mathbb{Z}^m$ (see Theorem 2.3 in \cite{BFP}).

In addition, for $\mu\in\mathcal{E}_{\{\omega\}}'(\mathbb{T}^{n+m}),$ we have \[\mu=\sum_{\xi\in\mathbb{Z}^m}\hat{\mu}(t,\xi)\otimes\exp(i\langle x,\xi\rangle),\] in which the sequence $\hat{\mu}(t,\xi)\subset\mathcal{E}_{\{\omega\}}'(\mathbb{T}^{n})$ satisfies: given $\epsilon>0$ and $k\in\mathbb{N}$ there exists $C_{\epsilon,k}>0$ such that 
\[|\langle\hat{\mu}(t,\xi),\theta(t)\rangle|\leq C_{\epsilon,k}\|\theta\|_{1/k}\exp\{\epsilon\omega(|\xi|)\},\] for all $\xi\in\mathbb{Z}^m$ and for all $\theta\in\mathcal{E}_{\{\omega\},k}(\mathbb{T}^{n}).$

\section{Global $\{\omega\}-$solvability - necessary conditions}\label{sec4}

In this section we present results concerning the existence of solutions to the equation $Lu=f.$ 

The first result shows that $L$ is not globally $\{\omega\}-$solvable in the presence of a disconnected sublevel set $\Omega_r^{\xi}$ (see \eqref{eqr7}) and in the case of a non-quasianalytic weight function. The quasianalytic version will not be treated here, since neither on the real analytic class the existence of solutions is completely understood in the presence of a disconnected sublevel set $\Omega_r^{\xi}.$  

\begin{proposition}\label{pror2} Suppose that $\{\omega\}$ is a non-quasianalytic subadditive weight function. If $b_{j0}=0,$ for all $j$ and at least one $b_k$ does not vanish identically, and $(a_{10},\ldots,a_{N0})\in\mathbb{Z}^N,$ then the existence of a disconnected sublevel set $\Omega_r^{\xi}$ implies that $L,$ given by \eqref{L_introduction}, is not globally $\{\omega\}-$solvable.
\end{proposition}

\begin{proof} The proof is by contradiction. We assume that $L$ is globally $\{\omega\}-$solvable and from Lemma \ref{lr1} we obtain, for each $h>0$ and $k>0,$ the existence of a constant $C_{h,k}>0$ satisfying
\begin{equation}\label{eqr4}
\left|\int_{\mathbb{T}^{N+1}}fv\right|\leq C_{h,k}\|f\|_{1/h}\|{}^tLv\|_{1/k},
\end{equation}
for every $f\in(\ker{}^tL)^\circ\cap\mathcal{E}_{\{\omega\},h}(\mathbb{T}^{N+1})$ and $v\in\mathcal{E}_{\{\omega\}}(\mathbb{T}^{N+1})$ such that ${}^tLv\in\mathcal{E}_{\{\omega\},k}(\mathbb{T}^{N+1}).$

Using a disconnected sublevel set $\Omega_r^{\xi'},$ we will construct sequences of functions, $\{f_n\}$ and $\{v_n\},$ which substituted in \eqref{eqr4} will produce a contradiction. 

If the sublevel $\Omega_r^{\xi'}$ is disconnected, then (as in \cite{Ho} and \cite{BDGK}), we can find a real number $r_0<r$ such that
$\Omega_{r_0}^{\xi'}$ has two connected components with disjoint closure. Since $\{\omega\}$ is non-quasianalytic, we recall that there exists cutoff functions with arbitrarily small support (see Proposition 2.4 and Lemma 3.3 in \cite{BMT}). We then pick $f_0,v_0\in \mathcal{E}_{\{\omega\}}(\mathbb{T}^1),$ satisfying the following conditions:
\[\int_{0}^{2\pi}f_0(t)dt=0,\ supp(f_0)\cap\Omega_{r_0}^{\xi'}=\emptyset,\]
\[supp(v_0')\subset\Omega_{r_0}^{\xi'},\] and
\[\int_{0}^{2\pi}f_0(t)v_0(t)dt>0.\]

Pick $\epsilon>0$ satisfying \[M\doteq\max_{t\in supp \,v_0' }\left\{\int_0^t\langle\xi',\beta(\tau)d\tau\rangle\right\}<\epsilon<r_0.\]

For $n\in\mathbb{N}$, we define the functions $f_n,v_n:\mathbb{T}^{N+1}\rightarrow\mathbb{C}$ by
\[
f_n(t,x)=\exp\left\{\epsilon n-n\int_{0}^{t}\langle\xi',\beta(\tau)\rangle d\tau\right\}\exp\left\{in\int_{0}^{t}\langle\xi',\alpha(\tau)\rangle d\tau\right\}f_0(t)e^{-in\langle\xi',x\rangle}, 
\] and
\[
v_n(t,x)=\exp\left\{-\epsilon n+n\int_{0}^{t}\langle\xi',\beta(\tau)\rangle d\tau\right\}\exp\left\{-in\int_{0}^{t}\langle\xi',\alpha(\tau)\rangle d\tau\right\}v_0(t)e^{in\langle\xi',x\rangle},
\] where we recall that $\langle\xi',\beta(\tau)\rangle=\sum_{j=1}^Nb_j(\tau)\xi'_j,$
$\langle\xi',\alpha(\tau)\rangle=\sum_{j=1}^Na_j(\tau)\xi'_j.$ 

Since $a_{j0}\in\mathbb{Z}$ and $b_{j0}=0,$ for all $j,$ it follows that $f_n$ and $v_n$ belong to $\mathcal{E}_{\{\omega\}}(\mathbb{T}^{N+1}).$ 

We claim that there exist positive constants $C$ and $h,$ which does not depend on $n,$ such that  $(f_n)\subset \mathcal{E}_{\{\omega\},h}(\mathbb{T}^{N+1})$ and \begin{equation}\label{eqr1}\|f_n\|_{1/h}\leq Ce^{-\frac{n(r_0-\epsilon)}{4}}.\end{equation}

In fact, pick $h_1>0$ such that $f_0\in\mathcal{E}_{\{\omega\},h_1}(\mathbb{T}^{1}).$ Then there exists $C_1>0$ such that 
\[|\partial_t^k\partial_x^{\gamma}f_n(t,x)|\leq |\xi'|^{|\gamma|} n^{|\gamma|}\sum_{m=0}^{k}\binom{k}{m}\left|\partial_t^m\left(e^{\epsilon n-n\int_{0}^{t}\langle\xi',\beta(\tau)-i\alpha(\tau)\rangle d\tau}\right)\right|C_1e^{h_1^{-1}\varphi^\ast(h_1(k-m))}.\]

Setting \[E_n(t)=\exp\left\{\epsilon n-n\int_{0}^{t}\langle\xi' ,\beta(\tau)-i\alpha(\tau)\rangle d\tau\right\},\] on $\mbox{supp}f_0$ we have $|E_n(t)|\leq e^{-n(r_0-\epsilon)},$ and for $m\geq1$ Fa\`a di Bruno's formula gives 
\[|\partial_t^mE_n(t)|\leq\sum_{k_1+2k_2+\cdots+mk_m=m}\frac{m!}{k_1!\cdots k_m!}e^{-(r_0-\epsilon)n}\prod_{\ell=1}^m\left|\frac{n\langle\xi',\partial_t^{\ell-1}[i\alpha(t)-\beta(t)]\rangle}{\ell!}\right|^{k_\ell}.\]

By using that $a_j,b_j\in\mathcal{E}_{\{w\}}(\mathbb{T}^1),$ we obtain positive constants $h_2$ and $C_2$ such that 
\[|\partial_t^{m}(ia_j(t)-b_j(t))|\leq C_2\exp\{h_2^{-1}\varphi^\ast(mh_2)\}\leq C_2,\] for all $t\in\mathbb{T}^1,$ $m\in\mathbb{Z}_+$ and $j=1,\ldots,N.$

Hence, \[|\partial_t^mE_n(t)|\leq\sum_{k_1+2k_2+\cdots+mk_m=m}\frac{m!}{k_1!\cdots k_m!}n^{k}e^{-(r_0-\epsilon)n}(|\xi'|NC_2)^k\prod_{\ell=1}^m\frac{e^{k_\ell h_2^{-1}\varphi^\ast((\ell-1)h_2)}}{\ell!^{k_\ell}},\] where $k=k_1+\cdots+k_m.$

Since $\varphi^\ast$ is increasing, we have $\varphi^\ast((\ell-1)h_2)\leq \varphi^\ast(\ell h_2.)$ In addition, from \eqref{eqr2} we obtain \[\prod_{\ell=1}^{m}\frac{e^{k_\ell h_2^{-1}\varphi^\ast((\ell-1)h_2)}}{\ell!^{k_\ell}}\leq \prod_{\ell=1}^{m}\frac{e^{h_2^{-1}\varphi^\ast(k_\ell \ell h_2)}}{(k_\ell\ell)!}\leq \frac{e^{h_2^{-1}\varphi^\ast(mh_2)}}{m!}.\]

It follows that 
\[|\partial_t^mE_n(t)|\leq\sum_{k_1+2k_2+\cdots+mk_m=m}\frac{(|\xi'|NC_2)^kn^{k}}{k_1!\cdots k_m!}e^{-(r_0-\epsilon)n}e^{h_2^{-1}\varphi^\ast(mh_2)}.\]

Since $t=2k/(r_0-\epsilon)$ is a maximum point of $t\mapsto t^ke^{-(r_0-\epsilon)t/2}$ on $(0,\infty),$ we pick $C_\epsilon=2/(r_0-\epsilon)$ to obtain 
\[n^ke^{\frac{-(r_0-\epsilon)n}{2}}\leq C_\epsilon^kk!.\]

In addition, applying Lemma 2.2 in \cite{BDG3} we obtain $R=|\xi'|NC_2C_\epsilon$ satisfying
\[|\partial_t^mE_n(t)|\leq e^{\frac{-(r_0-\epsilon)n}{2}}e^{h_2^{-1}\varphi^\ast(mh_2)}R(R+1)^{m-1}.\] 

By using the above estimate, we obtain 
\[|\partial_t^k\partial_x^{\gamma}f(t,x)|\leq C_1|\xi'|^{|\gamma|}(R+1)^kn^{|\gamma|}e^{\frac{-n(r_0-\epsilon)}{2}}\sum_{m=0}^{k}\binom{k}{m}e^{h_2^{-1}\varphi^\ast(mh_2)}e^{h_1^{-1}\varphi^\ast(h_1(k-m))}.\]

The assumption $\omega(t)=O(t),$ as $t$ tends to $\infty,$ implies that there exists $M>0$ such that $\omega(n)\leq Mn.$ Picking $h_3>4M/(r_0-\epsilon),$ it follows that
\begin{align*}n^{|\gamma|}e^{\frac{-n(r_0-\epsilon)}{4}}e^{-h_3^{-1}\varphi^\ast(h_3|\gamma|)}=&\exp\left\{|\gamma|\log(n)-h_3^{-1}\varphi^\ast(h_3|\gamma|)-\frac{n(r_0-\epsilon)}{4}\right\}\\
\leq&\exp\left\{h_3^{-1}\omega(n)-\frac{n(r_0-\epsilon)}{4}\right\}\\
\leq&\exp\left\{h_3^{-1}Mn-\frac{n(r_0-\epsilon)}{4}\right\}\leq1.\end{align*}

It follows that
\begin{multline*}|\partial_t^k\partial_x^{\gamma}f(t,x)|e^{-h_3^{-1}\varphi^\ast(h_3|\gamma|)}\leq \\
C_1|\xi'|^{|\gamma|}(R+1)^ke^{\frac{-n(r_0-\epsilon)}{4}}\sum_{m=0}^{k}\binom{k}{m}e^{h_2^{-1}\varphi^\ast(mh_2)}e^{h_1^{-1}\varphi^\ast(h_1(k-m))}.\end{multline*}

By \eqref{eqr2} and since $\varphi^\ast(t)/t$ is increasing, for $h_4=\max\{h_1,h_2,h_3\}$ we obtain 
\[\frac{k!}{m!(k-m)!}e^{h_2^{-1}\varphi^\ast(mh_2)}e^{h_1^{-1}\varphi^\ast(h_1(k-m))}\leq k!\frac{e^{h_4^{-1}\varphi^\ast(mh_4)}}{m!}\frac{e^{h_4^{-1}\varphi^\ast(h_4(k-m))}}{(k-m)!}\leq e^{h_4^{-1}\varphi^\ast(kh_4)},\] and, consequently,
\[|\partial_t^k\partial_x^{\gamma}f(t,x)|e^{-h_4^{-1}\varphi^\ast(h_4|\gamma|)}\leq C_1|\xi'|^{|\gamma|}(R+1)^ke^{\frac{-n(r_0-\epsilon)}{4}}(k+1)e^{h_4^{-1}\varphi^\ast(kh_4)}.\]

By Lemma 5.9 in \cite{RS}, there exist positive constants $C_3$ and $h_5>h_4$ such that 
\[(R+1)^k(k+1)e^{h_4^{-1}\varphi^\ast(kh_4)}\leq (R+1)^k2^ke^{h_4^{-1}\varphi^\ast(kh_4)}\leq C_3e^{h_5^{-1}\varphi^\ast(kh_5)},\] \[|\xi'|^{|\gamma|}e^{h_4^{-1}\varphi^\ast(h_4|\gamma|)}\leq C_3e^{h_5^{-1}\varphi^\ast(|\gamma|h_5)},\] and, consequently,
\[|\partial_t^k\partial_x^{\gamma}f(t,x)|e^{-h_5^{-1}\varphi^\ast(h_5|\gamma|)}e^{-h_5^{-1}\varphi^\ast(h_5k)}\leq C_1C_3e^{\frac{-n(r_0-\epsilon)}{4}}.\]

We now use superadditivity of $\varphi^\ast$ in order to obtain \eqref{eqr1}. 

A similar inspection shows that there exist positive constants $K$ and $\tilde{h},$ which does not depend on $n,$ such that  $({}^tLv_n)\subset \mathcal{E}_{\{\omega\},\tilde{h}}(\mathbb{T}^{N+1})$ and \begin{equation}\label{eqr3}\|{}^tLv_n\|_{1/\tilde{h}}\leq Ke^{-\frac{n(r_0-\epsilon)}{4}}.\end{equation}

We now show that $f_n\in(\ker{}^tL)^{\circ}.$ Indeed, if $\mu\in\ker{}^tL\subset\mathcal{E}_{\{\omega\}}'(\mathbb{T}^{N+1}),$ partial Fourier series in the variables $(x_1,\ldots,x_N)$ gives \[\mu=\sum_{\xi\in\mathbb{Z}^N}c_{\xi}\exp\left\{\int_{0}^{t}\langle\xi,\beta(\tau)-i\alpha(\tau)\rangle d\tau\right\}e^{i\langle\xi,x\rangle}, \,\,\textrm{with } c_{\xi}\in\mathbb{C}.\] 

Finally, applying \eqref{eqr4}, \eqref{eqr1} and \eqref{eqr3}, we obtain a positive constant $C_{h,\tilde{h}}$ such that 
\[0<(2\pi)^N\int_{0}^{2\pi}f_0(t)v_0(t)dt=\left|\int_{\mathbb{T}^{N+1}}f_nv_n\right|\leq C_{h,\tilde{h}}e^{-(r_0-\epsilon)n/2},\] for all $n\in\mathbb{Z}_+.$
Since $\displaystyle\lim_{n\rightarrow\infty} e^{-(r_0-\epsilon)n/2}=0,$ we obtain a contradiction.

Therefore, $L$ is not globally $\{\omega\}-$solvable.
\end{proof}

\begin{remark}As in \cite{BDG1,BDG3}, it is possible to find a conjugation $T$ such that the real part of the coefficients of the vector field $T\circ L\circ T^{-1}$  are constants. In addition, if the mean of the real part of one coefficient is an integer number, then we may assume that this real part is zero. This procedure simplify the proofs, but it is not crucial. Furthermore, in general this procedure may not be used for higher-order operators, neither for vector fields on others manifolds. With this in mind, we decided do not use this reduction throughout this article.
\end{remark}

Next result yields that $L$ is not globally $\{\omega\}-$solvable under the following situation: a certain $b_k$ does not vanish identically, all means $b_{j0}$ are zero, and at least one mean $a_{\ell0}$ is not an integer number.

\begin{proposition}\label{pror1} Let $L$ be given by \eqref{L_introduction}. Suppose that $b_{j0}=0$, for all $j,$ and that at least one $b_j$ does not vanish identically. If $(a_{10},\ldots,a_{N0})\not\in\mathbb{Z}^N,$ then $L$ is not globally $\{\omega\}-$solvable.
\end{proposition}

\begin{proof}The proof is inspired in techniques appearing in \cite{AD,B,BDG1,BDG3}. 

Without loos of generality, we may assume that one of the following two situations occurs:
\begin{itemize}
\item $b_1^{-1}(0)\neq\mathbb{T}^1,$ $b_{10}=0,$ $a_{10}\in\mathbb{Z}$ and $a_{20}\not\in\mathbb{Z}.$ 
\item $b_1^{-1}(0)\neq\mathbb{T}^1,$ $b_{10}=0$ and $a_{10}\not\in\mathbb{Z}.$
\end{itemize}

In both cases, we will construct a function \[f(t,x)=\sum_{n=1}^{\infty}\hat{f}(t,\xi(n))e^{i\langle\xi(n),x\rangle}\] satisfying $f\in(\ker{}^tL)^\circ\setminus L\mathcal{E}_{\{\omega\}}(\mathbb{T}^{N+1}).$ 

In each of the two above situations the proof is quite similar. We will be concentrate in the first situation and in the sequel we comment which changes are sufficient to prove the second situation. 

We also split the proof into the quasianalytic and non-quasianalytic cases:

\medskip

\textit{Case 1:} $\{\omega\}$ \textit{is a quasianalytic weight function}.

Pick $\xi(n)=k_n(1,1,0,\ldots,0)\in\mathbb{Z}^N,$ with $(k_n)\subset\mathbb{N}$ an increasing sequence satisfying $k_n\geq n$ and $a_{20}k_n\not\in\mathbb{Z},$ for all $n\in\mathbb{N}.$ 

The partial Fourier coefficients of $f$ will be of the form \[\hat{f}(t,\xi(n))=\tilde{E}_ne^{-k_n\psi(t)},\] in which \[\tilde{E}_n=1-e^{-2\pi i k_n(a_{10}+a_{20})}\] and $\psi\in\mathcal{E}_{\{\omega\}}(\mathbb{T}^{1})$ will be constructed in the sequel.

Since $b_1$ changes sign, we have $\emptyset\neq b_1^{-1}(0)\neq\mathbb{T}^1.$ Since $\omega$ is quasianalytic, it follows that the zeros of $b_1$ are isolated. By performing a translation in the variable $t,$ we may assume that $b_1(0)=0,$ $b_1<0$ on a small interval $[-\epsilon,0),$ and $b_1>0$ on $(0,\epsilon].$

Set \[A(t)=\int_{0}^{t}(a_1(s)+a_2(s))ds-t(a_{10}+a_{20})\] and \[B(t)=\int_{0}^{t}b_1(s)ds.\] It follows that $A$ and $B$ belong to $\mathcal{E}_{\{\omega\}}(\mathbb{T}^{1}),$ with $B(0)=0$ and $B(2\pi)<B(2\pi-\epsilon).$

Pick $t_0\in(0,2\pi)$ such that \[M\doteq B(t_0)=\max_{t\in[0,2\pi]}B(t)>0\] and define $\psi:\mathbb{T}^1\rightarrow\mathbb{C}$ by
\[\psi(t)=M+K(1-\cos(t))+i[(a_{1}(0)+a_{2}(0))\sin(t)-A(t_0)],\] in which $K>0$ is a constant that will be adjusted later.

Notice that $\psi$ is real-analytic. It follows that each $\hat{f}(\cdot,\xi(n))$ belongs to $\mathcal{E}_{\omega}(\mathbb{T}^1).$ 

As in \cite{BDG3}, we use the Fa\`a di Bruno formula to show that $f\in\mathcal{E}_{\{\omega\}}(\mathbb{T}^N).$

Notice that \begin{align*}|\partial_t^\ell\hat{f}(t,\xi(n))|\leq &\sum_{m_1+2m_2+\dots+\ell m_\ell=\ell}\frac{\ell !}{m_1!m_2!\cdots m_\ell!}e^{-k_n\Re\psi(t)}\prod_{j=1}^{\ell}k_n^{m_j}\left(\frac{|\psi^{(j)}(t)|}{j!}\right)^{m_j}\\
\leq&\sum_{m_1+2m_2+\dots+\ell m_\ell=\ell}\frac{\ell !}{m_1!m_2!\cdots m_\ell!}e^{-k_nM}k_n^{m}\prod_{j=1}^{\ell}\left(\frac{|\psi^{(j)}(t)|}{j!}\right)^{m_j},\end{align*} in which $m=m_1+\cdots+m_\ell.$

Picking $C_1>0$ and $h_1>0$ such that $|\psi^{(j)}(t)|\leq C_1e^{h_1^{-1}\varphi^\ast(jh_1)},$ for all $t\in[0,2\pi]$ and $j\in\mathbb{Z}_{+},$ we obtain
\[\prod_{j=1}^{\ell}\left(\frac{|\psi^{(j)}|(t)}{j!}\right)^{m_j}\leq C_1^{m}\prod_{j=1}^{\ell}\frac{e^{m_jh_1^{-1}\varphi^\ast(jh_1)}}{(j!)^{m_j}}.\]

As in the proof of Proposition \ref{pror2}, by applying \eqref{eqr2} it follows that 
\[\prod_{j=1}^{\ell}\frac{e^{m_jh_1^{-1}\varphi^\ast(jh_1)}}{(j!)^{m_j}}\leq\frac{e^{h_1^{-1}\varphi^\ast(\ell h_1)}}{(\ell!)}.\]

In addition, as before we have $e^{-k_nM/2}k_n^{m}\leq (2/M)^{m}m!,$ for all $n.$ 

Summarizing, we obtain 
\[|\partial_t^\ell\hat{f}(t,\xi(n))|\leq e^{-k_nM/2}e^{h_1^{-1}\varphi^\ast(\ell h_1)}\sum_{m_1+2m_2+\dots+\ell m_\ell=\ell}\frac{(2C_1/M)^mm!}{m_1!m_2!\cdots m_\ell!}.\]

By Lemma 2.2 in \cite{BDG3}, we obtain

\[|\partial_t^\ell\hat{f}(t,\xi(n))|\leq e^{-k_nM/2}(2C_1/M)(2C_1/M+1)^\ell e^{h_1^{-1}\varphi^\ast(\ell h_1)}.\]

It follows from Lemma 5.9 in \cite{RS} the existence of positive constants $C_2$ and $h_2>h_1$ satisfying \[|\partial_t^\ell\hat{f}(t,\xi(n))|\leq C_2e^{-k_nM/2}e^{h_2^{-1}\varphi^\ast(\ell h_2)}.\] 

Since $\omega(t)=O(t),$ as $t\rightarrow\infty,$ we may obtain $\gamma>0$ small enough such that \[-\frac{M}{2}k_n+\gamma\omega(|\xi(n)|)=-\frac{M}{2}k_n+\gamma\omega(2k_n)\leq 0,\] for all $n.$  

Hence, \[|\partial_t^\ell\hat{f}(t,\xi(n))|\leq C_2e^{h_2^{-1}\varphi^\ast(\ell h_2)}e^{-\gamma\omega(|\xi(n)|)}.\] 

As presented in Section \ref{sec3}, the above control on the partial Fourier coefficients implies that $f\in\mathcal{E}_{\{\omega\}}(\mathbb{T}^N).$

Notice also that $f\in(\ker{}^tL)^\circ,$ since for any $\mu\in\mathcal{E}_{\{\omega\}}'(\mathbb{T}^N)$ such that ${}^tL\mu=0,$ we have $\hat{\mu}(\cdot,-\xi(n))=0.$  

The next step in the proof is to show that $f\not\in L\mathcal{E}_{\{\omega\}}(\mathbb{T}^N).$

If there exists $u\in\mathcal{E}_{\{\omega\}}(\mathbb{T}^N)$ such that $Lu=f,$ then the partial Fourier series gives 
\[\partial_t\hat{u}(t,\xi(n))+k_n[-b_1(t)+i(a_1(t)+a_2(t))]\hat{u}(t,\xi(n))=\hat{f}(t,\xi(n))\] and, consequently, 
\begin{align*}\hat{u}(t,\xi(n))=&\tilde{E}_n^{-1}\int_{0}^{2\pi}\hat{f}(t-s,\xi(n))e^{k_n\int_{t-s}^{t}b_1(\tau)d\tau}e^{-ik_n\int_{t-s}^{t}(a_1+a_2)(\tau)d\tau}ds\\
=&\int_{0}^{2\pi}e^{-k_n\psi(t-s)}e^{k_n(B(t)-B(t-s))}e^{-ik_n\int_{t-s}
^{t}(a_1+a_2)(\tau)d\tau}ds.
\end{align*} 

In addition, by setting 
\begin{multline*}H(t,s)=B(t)-B(t-s)-M-K(1-\cos(t-s))+\\i[A(t-s)-A(t)-s(a_{10}+a_{20})
+A(t_0)-\sin(t-s)(a_1(0)+a_2(0))], \ \ \ t,s\in[0,2\pi],\end{multline*} we obtain
\[\hat{u}(t,\xi(n))=\int_{0}^{2\pi}e^{k_nH(t,s)}ds.\]

In the sequel, we will control the decaying of the sequence $\hat{u}(t_0,\xi(n)).$ 

We have \begin{multline*}H(t_0,s)=-B(t_0-s)-K(1-\cos(t_0-s))+\\
i[A(t_0-s)-s(a_{10}+a_{20})
-\sin(t_0-s)(a_1(0)+a_2(0))]\end{multline*} and
\[\hat{u}(t_0,\xi(n))=\int_{t_0-2\pi}^{t_0}e^{k_nH(t_0,t_0-\sigma)}d\sigma.\]

It follows that \begin{equation}\label{eqr5}|\hat{u}(t_0,\xi(n))|=\left|\int_{t_0-2\pi}^{t_0}e^{k_n[H(t_0,t_0-\sigma)+it_0(a_{10}+a_{20})]}d\sigma\right|.\end{equation}

Note that \[\Re H(t_0,t_0-\sigma)=-B(\sigma)-K(1-\cos(\sigma))\] and 
\[\Im H(t_0,t_0-\sigma)=A(\sigma)+(\sigma-t_0)(a_{10}+a_{20})
-\sin(\sigma)(a_1(0)+a_2(0)).\]

In order to have $\Re H(t_0,t_0-\sigma)<0,$ we use the conditions $0<t_0<2\pi$ and $B(0)=0$ to choose $K$ satisfying \[K> \sup_{\sigma\in[t_0-2\pi,t_0]}\frac{-B(\sigma)}{1-\cos(\sigma)}.\] 

The function \begin{align*}\phi(\sigma)\doteq& -iH(t_0,t_0-\sigma)+t_0(a_{10}+a_{20})\\
=&A(\sigma)+\sigma(a_{10}+a_{20})-(a_{1}(0)+a_2(0))\sin(\sigma)+i[B(\sigma)+K(1-\cos(\sigma))]\end{align*} satisfies the following conditions:
\[\phi(0)=0; \ \phi'(\sigma)=a_1(\sigma)+a_2(\sigma)-(a_1(0)+a_2(0))\cos(\sigma)+i(b_1(\sigma)+K\sin(\sigma));\] \[\phi'(0)=0; \ \textrm{and} \ \phi''(0)=a_1'(0)+a_2'(0)+i(b_1'(0)+K).\]

Increasing $K,$ we obtain the following three conditions: $\phi''(0)\neq0,$ $\phi'(\sigma)\neq0$ on a small set $(-\epsilon,\epsilon)\setminus\{0\},$ and $\Im \phi(\sigma)>0$ on $[t_0-2\pi,-\epsilon/2]\cup[\epsilon/2,t_0].$ Moreover, since $\sigma^2/|\phi'(\sigma)|^2$ tends to $[(a_1'(0)+a_2'(0))^2+(b_1'(0)+K)^2]^{-1},$ as $\sigma\rightarrow0,$ it follows that $|\sigma|/|\phi'(\sigma)|$ is bounded on a small interval $(-\epsilon,\epsilon).$ 

It follows that we may apply Theorem 7.7.5 in \cite{H1}, which yields the existence of a positive constant $C_1$ satisfying 
\[\left|\int_{|\sigma|<\epsilon}\chi(\sigma) e^{ik_n\phi(\sigma)}d\sigma-e^{ik_n\phi(0)}\left(\frac{k_n\phi''(0)}{2\pi i}\right)^{-1/2}\chi(0)\right|\leq \frac{C_1}{k_n}\sum_{|\alpha|\leq2}\sup_{\sigma\in[-\epsilon,\epsilon]}|D^\alpha\chi(\sigma)|,\] in which $\chi\in\mathcal{C}^\infty$ satisfies $\chi=1$ on $[-\epsilon/2,\epsilon/2]$ and $\textrm{supp }\chi\subset[-\epsilon,\epsilon].$  

Hence, there exists a complex number $C_2\neq0$ and a positive constant $C_3$ such that \begin{equation}\label{eqr6}\left|\int_{|\sigma|<\epsilon}\chi(\sigma) e^{ik_n\phi(\sigma)}d\sigma-\frac{C_2}{\sqrt{k_n}}\right|\leq \frac{C_3}{k_n}.\end{equation}

On the other hand \[\left|\int_{\sigma\in[t_0-2\pi,-\epsilon/2]\cup[\epsilon/2,t_0]}e^{ik_n\phi(\sigma)}d\sigma\right|\leq 2\pi e^{-k_n\mu},\] with \[0<\mu=\min\{B(\sigma)+K(1-\cos(\sigma)); \ \sigma\in[t_0-2\pi,-\epsilon/2]\cup[\epsilon/2,t_0]\}.\]

By the above estimate we obtain
\[\left|\int_{|\sigma|<\epsilon}\chi(\sigma) e^{ik_n\phi(\sigma)}d\sigma-\frac{C_2}{\sqrt{k_n}}\right|\geq\frac{|C_2|}{\sqrt{k_n}}-\left|\int_{|\sigma|<\epsilon/2}\chi(\sigma) e^{ik_n\phi(\sigma)}d\sigma\right|-2\pi e^{-k_n\mu}.\]

From \eqref{eqr5} and by the above estimate we obtain \begin{align*}|\hat{u}(t_0,\xi(n))|=&\left|\int_{t_0-2\pi}^{t_0} e^{ik_n\phi(\sigma)}d\sigma\right|\\
\geq&\left|\int_{|\sigma|<\epsilon/2} e^{ik_n\phi(\sigma)}d\sigma\right|-\left|\int_{\sigma\in[t_0-2\pi,-\epsilon/2]\cup[\epsilon/2,t_0]} e^{ik_n\phi(\sigma)}d\sigma\right|\\
\geq&\left|\int_{|\sigma|<\epsilon/2} e^{ik_n\phi(\sigma)}d\sigma\right|-2\pi e^{-k_n\mu}\\
\geq&\frac{|C_2|}{\sqrt{k_n}}-\left|\int_{|\sigma|<\epsilon}\chi(\sigma) e^{ik_n\phi(\sigma)}d\sigma-\frac{C_2}{\sqrt{k_n}}\right|-4\pi e^{-k_n\mu}.
\end{align*}

By \eqref{eqr6} we obtain
\[|\hat{u}(t_0,\xi(n))|\geq \frac{|C_2|}{\sqrt{k_n}}-\frac{C_3}{k_n}-4\pi e^{-k_n\mu}.\]

Hence, for $n$ sufficiently large we obtain \[|\hat{u}(t_0,\xi(n))|\geq \frac{|C_2|}{2\sqrt{k_n}}.\]

Since $u\in\mathcal{E}_{\{\omega\}}(\mathbb{T}^{N+1}),$ there exist positive constants $C_4$ and $\gamma$ such that
\[0<\frac{|C_2|}{2\sqrt{k_n}}\leq C_4 e^{-\gamma\omega(|\xi(n)|)}=C_4 e^{-\gamma\omega(2k_n)}.\]

Finally, since $\log(t)=o(\omega(t)),$ for $n$ large enough we obtain
\[0<\frac{|C_2|}{C_4}\leq\frac{1}{\sqrt{k_n}}.\] 

The above estimate produces a contradiction, since $k_n\rightarrow\infty$ as $n\rightarrow\infty.$ 

Therefore, cannot exist $u$ such that $Lu=f.$

This completes the proof in the first situation.

In the second situation, we pick $\xi(n)=k_n(1,0,\ldots,0),$ in which $k_na_{10}\not\in\mathbb{Z},$ for all $n.$ 

We then consider \[\hat{f}(t,\xi(n))=\tilde{\tilde{E}}_ne^{-k_n\tilde{\psi}(t)},\] in which \[\tilde{\tilde{E}}_n=1-e^{-2\pi ik_na_{10}}\] and \[\tilde{\psi}(t)=M+K(1-\cos(t))+i[a_1(0)\sin(t)-\tilde{A}(t_0)],\] with \[\tilde{A}(t)=\int_{0}^{t}a_{1}(s)ds-ta_{10}.\]

With these settings, it is enough to proceed as in the first situation to show that $f\in(\ker{}^tL)^\circ\setminus L\mathcal{E}_{\{\omega\}}(\mathbb{T}^{N+1}).$

This completes the proof in the quasianalytic case.

\medskip

\textit{Case 2:} $\omega$ \textit{is a non-quasianalytic weight function}.

Back to the first situation, we pick $\xi(n)=k_n(1,1,0,\ldots,0)\in\mathbb{Z}^N,$ with $(k_n)\subset\mathbb{N}$ an increasing sequence satisfying $k_n\geq n$ and $a_{20}k_n\not\in\mathbb{Z},$ for all $n\in\mathbb{N}.$

In order to construct the Fourier coefficients of $f,$ we now set \[M^\ast\doteq\max\left\{B(t)-B(t-s)=\int_{t-s}^t b_1(\tau)d\tau; \ 0\leq t,s\leq 2\pi\right\}=B(t_1)-B(t_1-s_1).\]

Since $b_1$ changes sign and $b_{10}=0,$ we have $M^\ast>0$ and $0<s_1<2\pi.$ 

With a translation in the variable $t,$ we may assume that $s_1,$ $t_1$ and $\sigma_1\doteq t_1-s_1$ belong to $(0,2\pi).$ 

Since $\omega$ is non-quasianalytic, we can pick $\Psi\in\mathcal{E}_{\{\omega\}}(\sigma_1-\epsilon,\sigma_1+\epsilon)$ such that $\Psi=1$ on a neighborhood of $[\sigma_1-\epsilon/2,\sigma_1+\epsilon/2]$ and $\mbox{supp} \Psi$ is a compact subset of $(\sigma_1-\epsilon,\sigma_1+\epsilon).$

We define $\hat{f}(t,\xi(n))$ as the $2\pi-$periodic extension of 
\[\Psi(t)e^{-ik_n\int_{t_1}^{t}(a_1+a_2)(\tau)d\tau}e^{-M^\ast k_n}, \ t\in[0,2\pi].\]   

Note that each $\hat{f}(\cdot,\xi(n))$ belongs to $\mathcal{E}_{\{\omega\}}(\mathbb{T}^1).$ In addition, setting \[E_n(t)=e^{-ik_n\int_{t_1}^{t}(a_1+a_2)(\tau)d\tau},\] we have 
\begin{equation}\label{eqr8}|\partial_t^m\hat{f}(t,\xi(n))|\leq e^{-M^\ast k_n}\left[\Psi^{(m)}(t)+\sum_{\ell=1}^{m}\binom{m}{\ell}|\Psi^{(m-\ell)}(t)||E_n^{(\ell)}(t)|\right].\end{equation}

We recall that there exit positive constants $C_1>1$ and $h_1$ such that $|\Psi^{(j)}(t)|\leq C_1e^{h_1^{-1}\varphi^\ast(jh_1)},$ for all $t\in[0,2\pi]$ and $j\in\mathbb{Z}_{+}.$

Similarly, there exit positive constants $C_2>1$ and $h_2$ such that $|(a_1+a_2)^{(j)}(t)|\leq C_2e^{h_2^{-1}\varphi^\ast(jh_2)},$ for all $t\in[0,2\pi]$ and $j\in\mathbb{Z}_{+}.$

Fa\`a di Bruno's formua gives 
\[|E_n^{(\ell)}(t)|\leq \sum_{m_1+2m_2+\cdots+\ell m_\ell=\ell}\frac{\ell!}{m_1!\cdots m_\ell!}\prod_{j=1}^{\ell}k_n^{m_j}\frac{|(a_1+a_2)^{(j-1)}(t)|^{m_j}}{j!^{m_j}}.\]

As before, since $\varphi^\ast$ is non-negative and superadditive, we have $\varphi^\ast((j-1)h_2)\leq \varphi^\ast(j h_2).$ In addition, by \eqref{eqr2} it follows that 
\begin{align*}\prod_{m=1}^{\ell}k_n^{m_j}\frac{|(a_1+a_2)^{(j-1)}(t)|^{m_j}}{j!^{m_j}}\leq & (C_2k_n)^{m_1+\cdots m_\ell}\prod_{j=1}^{\ell}\frac{e^{m_jh_2^{-1}\varphi^\ast((j-1)h_2)}}{j!^{m_j}}\\
\leq& (C_2k_n)^{m_1+\cdots m_\ell}\prod_{j=1}^{\ell}\frac{e^{m_jh_2^{-1}\varphi^\ast(jh_2)}}{j!^{m_j}}\\
\leq& (C_2k_n)^{m_1+\cdots m_\ell}\prod_{j=1}^{\ell}\frac{e^{h_2^{-1}\varphi^\ast(jm_j h_2)}}{(jm_j)!}\\
\leq&(C_2k_n)^{m_1+\cdots m_\ell}\frac{e^{h_2^{-1}\varphi^\ast(\ell h_2)}}{\ell!}
\end{align*}

Setting $k=m_1+\cdots m_\ell,$ we recall that $k_n^ke^{-M^\ast k_n/2}\leq (2/M^\ast)^kk!.$

The above estimates and Lemma 2.2 in \cite{BDG3} imply that
\begin{align*}|E_n^{(\ell)}(t)|e^{-M^\ast k_n/2}\leq &e^{h_2^{-1}\varphi^\ast(\ell h_2)}\sum_{m_1+2m_2+\cdots+\ell m_\ell=\ell}\frac{(2C_2/M^\ast)^kk!}{m_1!\cdots m_\ell!}\\
\leq& \left(\frac{2C_2}{M^\ast}+1\right)^{\ell}e^{h_2^{-1}\varphi^\ast(\ell h_2)}.
\end{align*}


From \eqref{eqr8} and the above estimate we obtain 
\[|\partial_t^m\hat{f}(t,\xi(n))|\leq C_1\left(\frac{2C_2}{M^\ast}+1\right)^me^{-\frac{M^\ast k_n}{2}}\sum_{\ell=0}^{m}\binom{m}{\ell}e^{h_1^{-1}\varphi^\ast((m-\ell)h_1)}e^{h_2^{-1}\varphi^\ast(\ell h_2)}\]

Picking $h_3=\max\{h_1,h_2\},$ since $\varphi^\ast(t)/t$ is increasing we obtain 
\[e^{h_1^{-1}\varphi^\ast((m-\ell)h_1)}e^{h_2^{-1}\varphi^\ast(\ell h_2)}\leq e^{h_3^{-1}\varphi^\ast((m-\ell)h_3)}e^{h_3^{-1}\varphi^\ast(\ell h_3)}.\]

Again, by using \eqref{eqr2} it follows that 
\[\frac{m!}{\ell!(m-\ell)!}e^{h_3^{-1}\varphi^\ast((m-\ell)h_3)}e^{h_3^{-1}\varphi^\ast(\ell h_3)}\leq e^{h_3^{-1}\varphi^\ast(mh_3)}.\]

Applying the above estimates we obtain
\begin{align*}|\partial_t^m\hat{f}(t,\xi(n))|\leq& C_1\left(\frac{2C_2}{M^\ast}+1\right)^me^{-M^\ast k_n/2}(m+1)e^{h_3^{-1}\varphi^\ast(mh_3)}.
\end{align*}

Lemma 5.9 in \cite{RS} gives positive constants $C_3>1$ and $h_4>h_3$ such that \[C_1\left(\frac{2C_2}{M^\ast}+1\right)^m(m+1)e^{h_3^{-1}\varphi^\ast(mh_3)}\leq C_3e^{h_4^{-1}\varphi^\ast(mh_4)}.\]

Hence, \[|\partial_t^m\hat{f}(t,\xi(n))|\leq C_3e^{-M^\ast k_n/2}e^{h_4^{-1}\varphi^\ast(mh_4)}.\]

In addition, since $\omega(t)=O(t),$ as $t\rightarrow\infty,$ we may obtain $\gamma^\ast>0$ such that \[-\frac{M^\ast}{2}k_n+\gamma^\ast\omega(|\xi(n)|)=-\frac{M^\ast}{2}k_n+\gamma^\ast\omega(2k_n)\leq 0,\] for all $n.$  

It follows that 
\[|\partial_t^m\hat{f}(t,\xi(n))|\leq C_3e^{h_4^{-1}\varphi^\ast(mh_4)}e^{-\gamma^\ast\omega(|\xi(n)|)}.\]

As presented in Section \ref{sec3}, the control on the partial Fourier coefficients of $f$ given by the above estimate implies that $f\in\mathcal{E}_{\{\omega\}}(\mathbb{T}^{N+1}).$

It is easy to see that $f\in(\ker{}^tL)^\circ,$ since the condition $k_n(a_{10}+a_{20})\not\in\mathbb{Z}$ implies that $\hat{\mu}(\cdot,-\xi(n))=0,$ for all $n$ and $\mu\in\ker{}^tL.$

We now proceed to show that $f\not\in L\mathcal{E}_{\{\omega\}}(\mathbb{T}^{N+1}).$ As before, if $u\in\mathcal{E}_{\{\omega\}}(\mathbb{T}^{N+1})$ and $Lu=f,$ then setting $\tilde{E_n}=1-e^{-2\pi i k_n(a_{10}+a_{20})}$ we have
\begin{align*}\hat{u}(t,\xi(n))=&\tilde{E}_n^{-1}\int_{0}^{2\pi}\hat{f}(t-s,\xi(n))e^{k_n\int_{t-s}^{t}b_1(\tau)d\tau}e^{-ik_n\int_{t-s}^{t}(a_1+a_2)(\tau)d\tau}ds\\
=&\int_{0}^{2\pi}\Psi(t-s)e^{k_n(B(t)-B(t-s)-M^\ast)}e^{-ik_n\int_{t_1}^{t}(a_1+a_2)(\tau)d\tau}ds.
\end{align*} 

In particular, 
\[|\hat{u}(t_1,\xi(n))|\geq\frac{1}{2}\int_{|s-s_1|<\epsilon/2}e^{-k_n(M^\ast+B(t_1-s)-B(t_1))}ds.\]

Notice that $s_1$ is a point of minimum of the function $M^\ast+B(t_1-s)-B(t_1).$ Hence, $s_1$ is either a zero of even order or a zero of infinite order. As in \cite{BDG1,BDG3}, we may apply the Lapace method for integrals on order to obtain a positive constant, does not depending on $n,$ such that \[|\hat{u}(t_1,\xi(n))|\geq \frac{C}{\sqrt{k_n}}.\] 

As in the \textrm{Case 1}, the above estimate produces a contradiction. Therefore, we conclude that $f$ cannot belong to $L\mathcal{E}_{\{\omega\}}(\mathbb{T}^{N+1}).$   
 
Finally, if we are in the second situation, then we pick $\xi(n)=k_n(1,0,\ldots,0)$ and $\hat{f}(t,\xi(n))$ is the $2\pi-$periodic extension of 
\[\Psi(t)e^{-ik_n\int_{t_1}^{t}a_1(\tau)d\tau}e^{-M^\ast k_n}, \ t\in[0,2\pi].\]   
 
In order to prove that $f$ belongs to $(\ker{}^tL)^\circ\setminus L\mathcal{E}_{\{\omega\}}(\mathbb{T}^{N+1}),$ it is enough to proceed as above.

\medskip

This completes the proof of Proposition \ref{pror1}. 
\end{proof}

We now proceed to establish other necessary conditions for the global $\{\omega\}-$solvability of $L.$

\begin{proposition}\label{pror3} The operator $L$ given by {\em(\ref{L_introduction})} is not globally $\{\omega\}-$solvable if $b_{j0}\neq0,$ for some $j=1,\ldots,N,$ and either $\dim \emph{\textrm{span}}\{b_1,\ldots,b_N\}>1$ or at least one $b_k(t)$ changes sign.
\end{proposition}
\begin{proof}The arguments are quite similar to those in Proposition \ref{pror1}. We will sketch some details.

Suppose firstly that $\dim\textrm{span}\{b_1,\ldots,b_N\}>1.$ In this case, there exist $m$ and $\ell$ in $\{1,\ldots,N\}$ such that $b_{m0}\neq0,$ and $b_m(t)$ and $b_\ell(t)$ are $\mathbb{R}-$linearly independent functions. There is no loss of generality in assuming that $m=1$ and $\ell=2.$

Lemma 3.1 of \cite{BDGK} implies the existence of integers $p$ and $q$ such that \[\mathbb{T}^1\ni t\mapsto\theta(t)=pb_1(t)+qb_2(t)\] changes sign and $\theta_0\doteq(2\pi)^{-1}\int_{0}^{2\pi}\theta(t)dt<0.$

Picking $\xi(n)=n(p,q,0,\ldots,0)\in\mathbb{Z}^N,$ $n\in\mathbb{Z}_{+},$ we will construct a function \begin{equation}\label{eqq12}f(t,x)=\sum_{n=1}^{\infty}\hat{f}(t,\xi(n))e^{i\langle\xi(n),x\rangle}\end{equation} satisfying $f\in(\ker{}^tL)^\circ\setminus L\mathcal{E}_{\{\omega\}}(\mathbb{T}^{N+1}).$

\medskip

\textit{Case 1:} $\{\omega\}$ \textit{is a quasianalytic weight function}.

We will describe the construction of the partial Fourier coefficients of $f.$ They will be of the form \[\hat{f}(t,\xi(n))=E_ne^{-n\psi(t)},\] in which \[E_n=1-e^{-2\pi i n(pa_{10}+qa_{20})}\] and $\psi\in\mathcal{E}_{\{\omega\}}(\mathbb{T}^{1})$ will be constructed in the sequel.

Since $\theta$ changes sign, we have $\emptyset\neq \theta^{-1}(0)\neq\mathbb{T}^1.$ Since $\omega$ is quasianalytic, it follows that the zeros of $\theta$ are isolated. As before, a translation in the variable $t$ allows us to assume that $\theta(0)=0,$ $\theta<0$ on a small interval $[-\epsilon,0),$ and $\theta>0$ on $(0,\epsilon].$

Set \[A(t)=\int_{0}^{t}(pa_1(s)+qa_2(s))ds-t(pa_{10}+qa_{20})\] and \[\Theta(t)=\int_{0}^{t}\theta(s)ds.\] It follows that $A$ belongs to $\mathcal{E}_{\{\omega\}}(\mathbb{T}^{1})$ and $\Theta$ is a smooth function satisfying $\Theta(0)=0$ and $\Theta(2\pi)<\Theta(2\pi-\epsilon).$ In contrast to the proof of Proposition \ref{pror1}, we stress that $\Theta(t)$ is not periodic, but it is not required in this proof.

Pick $t_0\in(0,2\pi)$ such that \[M\doteq \Theta(t_0)=\max_{t\in[0,2\pi]}\Theta(t)>0\] and define $\psi:\mathbb{T}^1\rightarrow\mathbb{C}$ by
\[\psi(t)=M+K(1-\cos(t))+i[(pa_{1}(0)+qa_{2}(0))\sin(t)-A(t_0)],\] in which $K>0$ is a constant. 

By proceeding as in the proof of Proposition \ref{pror1} we may adjust $K$ so that $f$ belong to $(\ker{}^tL)^\circ\setminus L\mathcal{E}_{\{\omega\}}(\mathbb{T}^{N+1}).$ Since this procedure is mutatis mutandis to the one in Proposition \ref{pror1}, we will omit the computations.

\medskip

\textit{Case 2:} $\omega$ \textit{is a non-quasianalytic weight function}.

Set \[M^\ast\doteq\max\left\{\Theta(t)-\Theta(t-s)=\int_{t-s}^t \theta(\tau)d\tau; \ 0\leq t,s\leq 2\pi\right\}=\Theta(t_1)-\Theta(t_1-s_1).\]

Since $\theta$ changes sign and $\theta_{0}<0,$ we have $M^\ast>0$ and $0<s_1<2\pi.$ 

As before, we may assume that $s_1,$ $t_1$ and $\sigma_1\doteq t_1-s_1$ belong to $(0,2\pi).$ 

We now pick $\Psi\in\mathcal{E}_{\{\omega\}}(\sigma_1-\epsilon,\sigma_1+\epsilon)$ such that $\Psi=1$ on a neighborhood of $[\sigma_1-\epsilon/2,\sigma_1+\epsilon/2]$ and $\mbox{supp} \Psi$ is a compact subset of $(\sigma_1-\epsilon,\sigma_1+\epsilon).$

Define $\hat{f}(t,\xi(n))$ as the $2\pi-$periodic extension of 
\[\Psi(t)e^{-in\int_{t_1}^{t}(pa_1+qa_2)(\tau)d\tau}e^{-M^\ast n}, \ t\in[0,2\pi].\]   

Each $\hat{f}(\cdot,\xi(n))$ belongs to $\mathcal{E}_{\{\omega\}}(\mathbb{T}^1).$ In addition, proceeding as in the proof of Proposition \ref{pror1} (\textit{Case 2}), we may show that $f\in(\ker{}^tL)^\circ\setminus L\mathcal{E}_{\{\omega\}}(\mathbb{T}^{N+1})\subset\mathcal{E}_{\{\omega\}}(\mathbb{T}^{N+1}).$ 

\bigskip

Assume now that one $b_k(t)$ changes sign and $\dim \textrm{span}\{b_1,\ldots,b_N\}=1.$ In this case, there exists $j$ such that $b_j$ changes sign and $b_{j0}\neq0.$ Without loos of generality, we may assume $j=1.$ In order to prove that $L$ is not globally $\{\omega\}-$solvable, it is enough to proceed as in the previous two cases, picking $\xi(n)=n(1,0\ldots,0).$  

\end{proof}

Next result completes the establishment of necessary conditions to the global $\{\omega\}-$solvability of $L.$ Its proof is an improvement of techniques in \cite{BDG1,BDG3} which enables us to treat simultaneously the quasianalytic and non-quasianalytic cases.

\begin{proposition}\label{pror4} The operator $L$ given by {\em(\ref{L_introduction})} is not globally $\{\omega\}-$solvable if $(\alpha_{0},\beta_{0})$ does not satisfy condition $(EDC)_{2}^{\{\omega\}}.$
\end{proposition}
\begin{proof} If $b_{j0}\neq0,$ for some $j=1,\ldots,N,$ then by Proposition \ref{pror3} we may assume that each $b_j$ does not change sign and $\dim \textrm{span}\{b_1,\ldots,b_N\}=1.$ Thus, we can write the operator $L$ as \begin{equation}\label{L_ldr}L=\frac{\partial}{\partial t}+\sum_{j=0}^{N}(a_{j}(t)+i\lambda_jb(t))\frac{\partial}{\partial x_j},\end{equation} where $b\in \mathcal{C}^{\infty}(\mathbb{T}^1,\mathbb{R}),$ $b$ does not change sign, $b_0\doteq(2\pi)^{-1}\int_{0}^{2\pi}b(t)dt\neq0,$ and $\lambda\doteq(\lambda_1,\ldots,\lambda_N)\in\mathbb{R}^N.$

If $b_{j0}=0,$ for each $j,$ and some $b_k$ does not vanish identically, then by Proposition \ref{pror1} we may assume that $\alpha_0=(a_{10},\ldots,a_{N0})\in\mathbb{Z}^N.$ Since $\beta_0=(b_{10},\ldots,b_{N0})=(0,\ldots,0),$ it follows that $(\alpha,\beta)$ satisfies $(EDC)_{2}^{\{\omega\}}.$ In other words, when $b_{j0}=0$ for each $j,$ then we may assume that each $b_j$ vanishes identically. In this case, the operator $L$ is again in the form \eqref{L_ldr}, with $\lambda=(0,\ldots,0).$

By the above comments, throughout this proof we assume that $L$ is given by \eqref{L_ldr}. Recall also that $\beta_0=(b_{10},\ldots,b_{N0})=b_0(\lambda_1,\ldots,\lambda_N)=b_0\lambda,$ and $\beta(t)=(b_1(t),\ldots,b_N(t))=b(t)(\lambda_1,\ldots,\lambda_N)=b(t)\lambda.$

By assumption, $(\alpha_0,\beta_0)=(\alpha_0,b_0\lambda)$ does not satisfy $(EDC)_{2}^{\{\omega\}}.$ Thus, there exist $\epsilon_0>0$ and a sequence $(\tau_n,\xi(n))\subset\mathbb{Z}\times\mathbb{Z}^N$ such that $|\tau_n|+|\xi(n)|\geq n$ and
\begin{equation}\label{eqq13}
0<|\tau_n+\langle\xi(n),\alpha_0+ib_0\lambda\rangle|<\exp\{-\epsilon_0\omega(|\tau_n|+|\xi(n)|)\},
\end{equation}for all $n\in\mathbb{N}.$

Set \[A_n(t)=\int_{0}^{t}\langle\xi(n),\alpha(s)\rangle ds-t\langle\xi(n),\alpha_0\rangle\] and \[B_n(t)=\int_{0}^{t}\langle\xi(n),b(s)\lambda\rangle ds-t\langle\xi(n),b_0\lambda\rangle.\] 

Estimate \eqref{eqq13} implies that the sequences of functions \[\exp\{B_n(t)\} \ \ \textrm{and} \ \ \exp\{-B_n(t)\}\] are bounded.

Let $f$ be given as in \eqref{eqq12}, with Fourier coefficients given by
\[\hat{f}(t,\xi(n))=\exp\{-\epsilon_0\omega(|\tau_n|+|\xi(n)|)/2\}\exp\{i\tau_n t\}\exp\{B_n(t)-iA_n(t)\}.\]

Each $f(\cdot,\xi(n))$ belongs to $\mathcal{E}_{\{\omega\}}(\mathbb{T}^1)$ and the term $\exp\{-\epsilon_0\omega(|\tau_n|+|\xi(n)|)/2\}$ will imply that $f\in\mathcal{E}_{\{\omega\}}(\mathbb{T}^{N+1}).$ We will sketch some details. Setting $c_n=\exp\{-\epsilon_0\omega(|\tau_n|+|\xi(n)|)/2\}$ and picking $C_1$ such that $e^{B_n(t)}\leq C_1,$ for all $n$ and $t,$ then Fa\`{a} Di Bruno's formula gives
\begin{multline*}|\partial_t^m\hat{f}(t,\xi(n))|\leq \sum_{k_1+\cdots+mk_m=m}\frac{m!C_1c_n}{k_1!\ldots k_m!}(i\tau_n+\langle\xi(n),\lambda(b(t)-b_0)+i(\alpha_0-\alpha(t))\rangle)^{k_1}\times\\
\prod_{j=2}^{m}\left(\frac{\langle\xi(n),b^{(j-1)}(t)\lambda-i\alpha^{(j-1)}(t)\rangle}{j!}\right)^{k_j}\end{multline*}

Since $b\in\mathcal{E}_{\{\omega\}}(\mathbb{T}^{1})$ and $\alpha(t)=(a_1(t),\ldots,a_{N}(t)),$ with $a_\ell\in\mathcal{E}_{\{\omega\}}(\mathbb{T}^{1}),$ we may obtain positive constants $C_2$ and $h_2,$ do no depending on $n$ or $j,$ such that 
\[|i\tau_n+\langle\xi(n),\lambda(b(t)-b_0)+i(\alpha_0-\alpha(t))\rangle|\leq (|\tau_n|+|\xi(n)|)C_2\] and \[|\langle\xi(n),b^{(j-1)}(t)\lambda-i\alpha^{(j-1)}(t)\rangle|\leq|\xi(n)|C_2e^{h_2^{-1}\varphi^{\ast}((j-1)h_2)},\] for all $n$ and $t.$

It follows that \[|\partial_t^m\hat{f}(t,\xi(n))|\leq C_1c_n\sum_{k_1+\cdots+mk_m=m}\frac{m!C_2^k(|\tau_n|+|\xi(n)|)^k}{k_1!\ldots k_m!}\prod_{j=1}^{m}\left(\frac{e^{h_2^{-1}\varphi^{\ast}((j-1)h_2)}}{j!}\right)^{k_j},\] in which $k=k_1+\cdots+k_m.$

As we have done previously, by using \eqref{eqr2} we obtain \[\prod_{j=1}^{m}\frac{e^{k_jh_2^{-1}\varphi^{\ast}((j-1)h_2)}}{j!^{k_j}}\leq\prod_{j=1}^{m} \frac{e^{h_2^{-1}\varphi^{\ast}(k_j(j-1)h_2)}}{[k_j(j-1)]!j^{k_j}}\leq \frac{e^{h_2^{-1}\varphi^{\ast}((m-k)h_2)}}{(m-k)!1^{k_1}2^{k_2}\ldots m^{k_m}}\] 

Hence, \begin{align*}|\partial_t^m\hat{f}(t,\xi(n))|\leq &C_1c_n\sum_{k_1+\cdots+mk_m=m}\frac{C_2^k(|\tau_n|+|\xi(n)|)^ke^{h_2^{-1}\varphi^{\ast}((m-k)h_2)}m!}{k_1!\ldots k_m!1^{k_1}2^{k_2}\ldots m^{k_m}(m-k)!}\\
\leq&C_1c_n2^m\sum_{k_1+\cdots+mk_m=m}\frac{C_2^k(|\tau_n|+|\xi(n)|)^ke^{h_2^{-1}\varphi^{\ast}((m-k)h_2)}k!}{k_1!\ldots k_m!1^{k_1}2^{k_2}\ldots m^{k_m}}.
\end{align*}

By Lemma 2.6 in \cite{BFP}, for each $K\in\mathbb{N}$ we obtain 
\[(|\tau_n|+|\xi(n)|)^ke^{-\omega(|\tau_n|+|\xi(n)|)/K}\leq e^{K^{-1}\varphi^{\ast}(kK)},\] for all $n$ and $k\in\mathbb{Z}_+.$ 

Picking $K\geq 4/\epsilon_0,$ it follows that \[(|\tau_n|+|\xi(n)|)^ke^{-\epsilon_0\omega(|\tau_n|+|\xi(n)|)/4}\leq e^{K^{-1}\varphi^{\ast}(kK)}.\] 

If $h_3=\max\{h_2,K\},$ then we obtain 
\begin{multline*}|\partial_t^m\hat{f}(t,\xi(n))|\leq \\
C_12^me^{-\epsilon_0\omega(|\tau_n|+|\xi(n)|)/4}e^{h_3^{-1}\varphi^{\ast}(mh_3)}\sum_{k_1+\cdots+mk_m=m}\frac{C_2^kk!}{k_1!\ldots k_m!1^{k_1}2^{k_2}\ldots m^{k_m}}
\end{multline*}

Lemma 2.2 in \cite{BDG3} gives
\[|\partial_t^m\hat{f}(t,\xi(n))|\leq
C_1(2C_2+2)^me^{-\epsilon_0\omega(|\tau_n|+|\xi(n)|)/4}e^{h_3^{-1}\varphi^{\ast}(mh_3)}\]

Applying Lemma 5.9 in \cite{RS}, we obtain $C_3>0$ and $h>0$ such that 
\[|\partial_t^m\hat{f}(t,\xi(n))|\leq C_3e^{-\epsilon_0\omega(|\xi(n)|)/4}e^{h^{-1}\varphi^{\ast}(mh)},\] for all $m,$ $t$ and $n.$ 

Estimate above implies that $f\in\mathcal{E}_{\{\omega\}}(\mathbb{T}^{N+1}).$

Since $\tau_n+\langle\xi(n),\alpha_0+ib_0\lambda\rangle\neq0,$ we obtain $f\in(\ker{}^tL)^{\circ}.$

In addition, if there exists $u\in \mathcal{E}_{\{\omega\}}(\mathbb{T}^{N+1})$ such that $Lu=f,$ then the function \[v(t,x)=\sum_{n=1}^{\infty}\hat{u}(t,\xi(n))e^{-B_n(t)+iA_n(t)}e^{i\langle\xi(n),x\rangle}\] belongs to $\mathcal{E}_{\{\omega\}}(\mathbb{T}^{N+1})$ and its partial Fourier coefficients satisfy \begin{align*}\partial_t\hat{v}(t,\xi(n))+\langle \xi(n),i\alpha_0-b_0\lambda\rangle\hat{v}(t,\xi(n))=&\hat{f}(t,\xi(n))e^{iA_n(t)-B_n(t)}\\
=&e^{-\epsilon_0\omega(|\tau_n|+|\xi(n)|)/2}e^{i\tau_n t}.\end{align*}

By using Fourier series in the variable $t,$ we write \[\hat{v}(t,\xi(n))=\sum_{\tau\in\mathbb{Z}}\hat{v}(\tau,\xi(n))e^{i\tau t}.\] 

It follows that $i(\tau_n+\langle\xi(n),\alpha_0+ib_0\lambda\rangle)\hat{v}(\tau_n,\xi(n))=e^{-\epsilon_0\omega(|\tau_n|+|\xi(n)|)/2}.$ Consequently, \[|\hat{v}(\tau_n,\xi(n))|\geq e^{-\epsilon_0\omega(|\tau_n|+|\xi(n)|)/2}|\tau_n+\langle\xi(n),\alpha_0+ib_0\lambda\rangle|^{-1}> e^{\epsilon_0\omega(|\xi(n)|)/2}\geq1,\] which is a contradiction, since $\hat{v}(\tau_n,\xi(n))$ must decay.

\medskip

The proof of Proposition \ref{pror4} is complete.
\end{proof}

\section{Global $\{\omega\}-$solvability - sufficient conditions}\label{sec5}

In the previous section we found necessary conditions to the global $\{\omega\}-$solvability of the operator $L$ given by \eqref{L_introduction}, which we recall \[L=\frac{\partial}{\partial t}+\sum_{j=1}^{N}(a_j+ib_j)(t)\frac{\partial}{\partial x_j}.\]

The purpose now is to complete the proof of Theorem \ref{mtr2} by showing sufficiency of these conditions.

\begin{proposition}\label{pror5} 
Suppose that $b_{j0}=0,$ for all $j=1,\ldots,N,$ and at least one $b_k$ does not vanish identically. Suppose also that $(a_{10},\ldots,a_{N0})\in\mathbb{Z}^N$ and that all the sublevel sets $\Omega_r^{\xi}$ are connected. Under these conditions, the operator $L$ given by \eqref{L_introduction} is globally $\{\omega\}-$solvable.
\end{proposition}

\begin{proof}Given $f\in(\ker{}^tL)^{\circ},$ we will show that there exists $u\in\mathcal{E}_{\{\omega\}}(\mathbb{T}^{N+1})$ satisfying $Lu=f.$ Writing \[u(t,x)=\sum_{\xi\in\mathbb{Z}^N}\hat{u}(t,\xi)e^{i\langle\xi,x\rangle},\] the Fourier coefficients $\hat{u}(t,\xi)$ must satisfy
\begin{equation}\label{eqq18}\partial_t\hat{u}(t,\xi)+i\langle\xi,\alpha(t)+i\beta(t)\rangle\hat{u}(t,\xi)=\hat{f}(t,\xi),\end{equation} in which we recall that $\alpha(t)=(a_{1}(t),\ldots,a_{N}(t))$ and $\beta(t)=(b_{1}(t),\ldots,b_{N}(t)).$

Since $(a_{10},\ldots,a_{N0})\in\mathbb{Z}^N$ and $b_{j0}=0,$ for all $j=1,\ldots,N,$ it follows that 
\begin{equation*}
\exp\left\{i\int_{0}^{t}\langle\xi,\alpha(r)+i\beta(r)\rangle dr\right\}\exp\{-i\langle\xi,x\rangle\}
\end{equation*}is a $2\pi-$periodic function which belongs to $\ker{}^tL.$ Since $f\in(\ker{}^tL)^\circ,$ we obtain 
\begin{equation}\label{eqq17}
\int_{0}^{2\pi}\hat{f}(t,\xi)e^{i\int_{0}^{t}\langle\xi,\alpha(r)+i\beta(r)\rangle dr}dt=0.
\end{equation}

Hence, equation \eqref{eqq18} has infinitely many solutions. Similar to which happens in Proposition 5.1 in \cite{BDG3}, these solutions are given in the generic form 
\[\hat{u}(t,\xi)=\int_{t_\xi}^{t}\hat{f}(s,\xi)e^{i\int_{t}^{s}\langle\xi,\alpha(r)+i\beta(r)\rangle dr}ds.\]

With a suitable choice of $t_\xi$ we will obtain an adequate decaying of the sequence $\hat{u}(t,\xi)$ to produce a solution $u$ in $\mathcal{E}_{\{\omega\}}(\mathbb{T}^{N+1}).$

Pick $t_\xi$ satisfying \[\int_{0}^{t_\xi}\langle\xi,\beta(r)\rangle dr=\sup_{t\in\mathbb{T}^1}\int_{0}^{t}\langle\xi,\beta(r)\rangle dr.\]

We must find positive constants $C,$ $h$ and $\epsilon$ such that 
\begin{equation}\label{eqq14}
|\partial_t^m\hat{u}(t,\xi)|\leq C\exp\{h^{-1}\varphi^\ast(mh)\}\exp\{-\epsilon\omega(|\xi|)\}, \ \ \textrm{for all} \ \ t \ \ , \ \ m \ \ \textrm{and} \ \ \xi.
\end{equation}

We may write 
\begin{align}\label{eqq15}
\partial_t^m\hat{u}(t,\xi)=&\partial_t^m\left(e^{-i\int_{0}^{t}\langle\xi,\alpha(r)+i\beta(r)\rangle dr}\right)\int_{t_\xi}^{t}\hat{f}(s,\xi)e^{i\int_{0}^{s}\langle\xi,\alpha(r)+i\beta(r)\rangle dr}ds\\
\label{eqq16}+&\sum_{n=1}^{m}\binom{m}{n}\partial_t^{n-1}\left(\hat{f}(t,\xi)e^{i\int_{0}^{t}\langle\xi,\alpha(r)+i\beta(r)\rangle dr}\right)\partial_t^{m-n}\left(e^{-i\int_{0}^{t}\langle\xi,\alpha(r)+i\beta(r)\rangle dr}\right).
\end{align}

It is enough to show that both \eqref{eqq15} and \eqref{eqq16} satisfy \eqref{eqq14}.

For each $t$ and $\xi,$ the sublevel set 
\[\Omega=\left\{s\in\mathbb{T}^1;-\int_{0}^{s}\langle\xi,\beta(r)\rangle dr\leq-\int_{0}^{t}\langle\xi,\beta(r)\rangle dr\right\}\] is connected. Since $t$ and $t_\xi$ belong to $\Omega,$ we may pick an arc $\Gamma\subset \Omega$ joining $t_\xi$ and $t.$ By \eqref{eqq17} we obtain 
\[\left|\int_{t_\xi}^{t}\hat{f}(s,\xi)e^{i\int_{0}^{s}\langle\xi,\alpha(r)+i\beta(r)\rangle dr}ds\right|=\left|\int_{\Gamma}\hat{f}(s,\xi)e^{i\int_{t}^{s}\langle\xi,\alpha(r)+i\beta(r)\rangle dr}ds\right|.\]

The above identity implies that 
\[e^{\int_{0}^{t}\langle\xi,\beta(r)\rangle dr}\left|\int_{t_\xi}^{t}\hat{f}(s,\xi)e^{i\int_{0}^{s}\langle\xi,\alpha(r)+i\beta(r)\rangle dr}ds\right|\leq 2\pi\max_{t\in[0,2\pi]}|\hat{f}(t,\xi)|\leq C_1e^{-\epsilon_1\omega(|\xi|)},\] in which $C_1$ and $\epsilon_1$ are positive constants which do not depend on $t$ and $\xi.$

By using Fa\`a di Bruno's formula, we obtain positive constants $C_2$ and $h_1$ such that 
\begin{multline*}\left|\partial_t^m\left(e^{-i\int_{0}^{t}\langle\xi,\alpha(r)+i\beta(r)\rangle dr}\right)\right|\leq\\
C_2^m\exp\{h_1^{-1}\varphi^\ast(mh_1)\}e^{\int_{0}^{t}\langle\xi,\beta(r)\rangle dr}\sum_{k_1+\cdots+mk_m=m}\frac{|\xi|^k}{k_1!\ldots k_m!}.\end{multline*}

Summarizing the last two estimates we obtain
\begin{multline*}\left|\partial_t^m\left(e^{-i\int_{0}^{t}\langle\xi,\alpha(r)+i\beta(r)\rangle dr}\right)\int_{t_\xi}^{t}\hat{f}(s,\xi)e^{i\int_{0}^{s}\langle\xi,\alpha(r)+i\beta(r)\rangle dr}ds\right|\leq\\
C_1e^{-\epsilon_1\omega(|\xi|)}C_2^m\exp\{h_1^{-1}\varphi^\ast(mh_1)\}\sum_{k_1+\cdots+mk_m=m}\frac{|\xi|^k}{k_1!\ldots k_m!}.\end{multline*}

Proceeding as in the proof of Proposition \ref{pror4} (applying Lemma 2.6 in \cite{BFP} and  Lemma 5.9 in \cite{RS}) we obtain the desired positive constants $C,$ $h$ and $\epsilon$ such that 
\begin{multline*}\left|\partial_t^m\left(e^{-i\int_{0}^{t}\langle\xi,\alpha(r)+i\beta(r)\rangle dr}\right)\int_{t_\xi}^{t}\hat{f}(s,\xi)e^{i\int_{0}^{s}\langle\xi,\alpha(r)+i\beta(r)\rangle dr}ds\right|\leq\\
Ce^{h^{-1}\varphi^\ast(mh)}e^{-\epsilon\omega(|\xi|)}.
\end{multline*}

In order to show that \eqref{eqq16} also satisfies \eqref{eqq14} we use a similar strategy. The computations are quite similar to the ones we have already performed in the proof of Proposition \ref{pror4}. We omit the details.

\end{proof}

\begin{remark}If $\alpha_0=(a_{10},\ldots,a_{N0})\in\mathbb{Z}^N$ and $\beta_0=(b_{10},\ldots,b_{N0})=(0,\ldots,0),$ then $(\alpha_{0},\beta_{0})$ satisfies condition $(EDC)_{2}^{\{\omega\}}.$ Hence, condition $(EDC)_{2}^{\{\omega\}}$ is satisfied in Proposition \ref{pror5}.
\end{remark}

Before we proceed, we state a technical result concerning the number theoretic condition $(EDC)_{2}^{\{\omega\}}.$ Its proof follows the same lines as in Lemma 3.1 in \cite{BDG1}. 

\begin{lemma}\label{lemrdc}A pair $(\alpha,\beta)\in\mathbb{R}^N\times\mathbb{R}^N$ satisfies condition $(EDC)_{2}^{\{\omega\}}$ if and only if for each $\epsilon>0$ there exists a positive constant $C_\epsilon$ such that 
\[|1-e^{-2\pi i\langle\xi,\alpha+i\beta\rangle}|\geq C_\epsilon e^{-\epsilon\omega(|\xi|)},\] for all $\xi\in\mathbb{Z}^N$ such that $\langle\xi,\alpha+i\beta\rangle\not\in\mathbb{Z}.$
\end{lemma}

\begin{proposition}\label{pror6}Suppose that $b_{j0}\neq0,$ for some $j=1,\ldots,N,$ and $b_k(t)$ does not change sign, for each $k,$ and $\dim \emph{\textrm{span}}\{b_1,\ldots,b_N\}=1.$ If $(\alpha_{0},\beta_{0})$ satisfies condition $(EDC)_{2}^{\{\omega\}},$ then the operator $L$ given by {\em(\ref{L_introduction})} is globally $\{\omega\}-$solvable.
\end{proposition}

\begin{proof}Under the assumptions, we may write \[L=\frac{\partial}{\partial t}+\sum_{j=0}^{N}(a_{j}(t)+i\lambda_jb(t))\frac{\partial}{\partial x_j},\] where $b\in \mathcal{C}^{\infty}(\mathbb{T}^1,\mathbb{R}),$ $b$ does not change sign, $b_0\doteq(2\pi)^{-1}\int_{0}^{2\pi}b(t)dt\neq0,$ and $\lambda\doteq(\lambda_1,\ldots,\lambda_N)\in\mathbb{R}^N\setminus\{0\}.$ The procedure is the same as in the previous proposition, that is, given $f\in(\ker{}^tL)^{\circ},$ we will show that there exists $u\in\mathcal{E}_{\{\omega\}}(\mathbb{T}^{N+1})$ satisfying $Lu=f.$ By using partial Fourier series, we will solve the equations
\begin{equation*}\label{eqq19}\partial_t\hat{u}(t,\xi)+i\langle\xi,\alpha(t)+ib(t)\lambda\rangle\hat{u}(t,\xi)=\hat{f}(t,\xi),\end{equation*} in which we recall that $\alpha(t)=(a_{1}(t),\ldots,a_{N}(t)).$ 

If $\xi\in\mathbb{Z}^N$ is such that $\langle\xi,\lambda\rangle=0$ and $\langle\xi,\alpha_0\rangle\in\mathbb{Z},$ then we choose the solution 
\[\hat{u}(t,\xi)=\int_{0}^{t}\hat{f}(s,\xi)e^{i\int_{t}^{s}\langle\xi,\alpha(r)\rangle dr}ds.\] 

To the other indices $\xi,$ the equation \eqref{eqq19} has a unique solution, which may be written as either

\[\hat{u}(t,\xi)=(1-e^{2\pi\langle\xi,b_0\lambda-i\alpha_0\rangle})^{-1}\int_{0}^{2\pi}\hat{f}(\xi,t-s)e^{\int_{t-s}^{t}\langle\xi,b(r)\lambda-i\alpha(r)\rangle dr}ds\]

or

\[\hat{u}(t,\xi)=(e^{-2\pi\langle\xi,b_0\lambda-i\alpha_0\rangle}-1)^{-1}\int_{0}^{2\pi}\hat{f}(\xi,t+s)e^{\int_{t}^{t+s}\langle\xi,-b(r)\lambda+i\alpha(r)\rangle dr}ds.\]

In order to show that the sequence of solutions above selected has the desired decaying, we proceed as in the proof of Proposition \ref{pror4} (applying Lemma 2.6 in \cite{BFP} and  Lemma 5.9 in \cite{RS}) in addition to Lemma \ref{lemrdc} and with a particular choice between the two above formulas (so that the exponential will be always bounded; for instance, if $\xi$ is such that $b(t)\langle\xi,\lambda\rangle\leq0,$ then we use the first formula).

The computations used to prove the decaying of the sequence of solutions do not bring any novelty to this paper and the details are omitted. 
\end{proof}

\begin{proposition}\label{pror7}Suppose that $b_{j}$ vanishes identically, for each $j.$ If $(\alpha_{0},\beta_{0})$ satisfies condition $(EDC)_{2}^{\{\omega\}},$ then the operator $L$ given by {\em(\ref{L_introduction})} is globally $\{\omega\}-$solvable.
\end{proposition}

\begin{proof}The assumptions reduce the operator $L$ to a real operator in the form 
\[L=\frac{\partial}{\partial t}+\sum_{j=0}^{N}a_{j}(t)\frac{\partial}{\partial x_j},\] and the partial Fourier series lead us to the equations 
\begin{equation*}\label{eqq20}\partial_t\hat{u}(t,\xi)+i\langle\xi,\alpha(t)\rangle\hat{u}(t,\xi)=\hat{f}(t,\xi).\end{equation*}

If $\xi\in\mathbb{Z}^N$ is such that $\langle\xi,\alpha_0\rangle\in\mathbb{Z},$ then we pick the solution 
\[\hat{u}(t,\xi)=\int_{0}^{t}\hat{f}(s,\xi)e^{i\int_{t}^{s}\langle\xi,\alpha(r)\rangle dr}ds.\] 

On the other hand, if $\langle\xi,\alpha_0\rangle\not\in\mathbb{Z},$ then we have a unique solution given by
\[\hat{u}(t,\xi)=(1-e^{-2\pi i\langle\xi,\alpha_0\rangle})^{-1}\int_{0}^{2\pi}\hat{f}(\xi,t-s)e^{-i\int_{t-s}^{t}\langle\xi,\alpha(r)\rangle dr}ds.\] 

As before condition $(EDC)_{2}^{\{\omega\}}$ implies that the selected sequence of solutions has the desired decaying.
\end{proof}

The results in this section together with the results in Section \ref{sec4} complete the proof of Theorem \ref{mtr2}.

\section{Strong $\{\omega\}-$solvability and Global $\{\omega\}-$hypoellipticity}\label{sec6}

This section is dedicated to prove Theorem \ref{mtr1}. 

We first recall that, as in the class of smooth functions and in the Gevrey classes, the operator $L$ is strongly $\{\omega\}-$solvable if and only if $L$ is globally $\{\omega\}-$solvable and $\ker {}^tL$ has finite dimension. 

The unique situation in which the dimension of $\ker {}^tL$ is infinite is in the presence of an index $\xi\in\mathbb{Z}^N\setminus\{0\}$ such that $\langle\xi,\beta_0\rangle=0$ and $\langle\xi,\alpha_0\rangle\in\mathbb{Z}.$ We also notice that condition $(EDC)_{2}^{\{\omega\}}$ is equivalent to $(EDC)_{1}^{\{\omega\}}$ whenever, for each $\xi\in\mathbb{Z}^N\setminus\{0\},$ either $\langle\xi,\beta_0\rangle\neq0$ or $\langle\xi,\alpha_0\rangle\not\in\mathbb{Z}.$

It follows from Proposition \ref{pror4} that $(EDC)_{1}^{\{\omega\}}$ is a necessary condition to the strong $\{\omega\}-$solvability.

If $\beta_0=(b_{10},\ldots,b_{N0})=0,$ then by using Proposition \ref{pror2} we see that $L$ is not strongly $\{\omega\}-$solvable if there exists $j$ such that $b_j$ does not vanish identically. On the other hand, if $b_k$ vanishes identically for each $k,$ by using item $(I.1)$ in Theorem \ref{mtr2} we see that $L$ is strongly $\{\omega\}-$solvable if and only if $(\alpha_0,0)$ satisfies $(EDC)_{1}^{\{\omega\}}.$     

If $b_{j0}\neq0$ for some $j,$ then Proposition \ref{pror3} implies that $L$ is not strongly $\{\omega\}-$solvable if $\dim \textrm{span}\{b_1,\ldots,b_N\}>1$ or at least one $b_k(t)$ changes sign. Finally, supposing that $(\alpha_0,\beta_0)$ satisfies $(EDC)_{1}^{\{\omega\}},$ then $\ker^tL$ has finite dimension and supposing that $b_{j0}\neq0$ for some $j,$ $\dim \textrm{span}\{b_1,\ldots,b_N\}\leq1$ and $b_k(t)$ does not change sign, for each $k,$ then Theorem \ref{mtr2} implies that $L$ is strongly $\{\omega\}-$solvable.

The above comments prove that the conditions $(1),$ $(2)$ and $(3)$ in Theorem \ref{mtr1} characterizes the  strong $\{\omega\}-$solvability of $L.$

\bigskip

In the sequel, we will complete the proof of Theorem \ref{mtr1} by treating the $\{\omega\}-$ hypoellipticity.

We start by proving that the conditions presented in Theorem \ref{mtr1} are sufficient to the global  $\{\omega\}-$hypoellipticity. Under such conditions, the operator $L$ has the form \[L=\frac{\partial}{\partial t}+\sum_{j=0}^{N}(a_{j}(t)+i\lambda_jb(t))\frac{\partial}{\partial x_j},\] where $b\in \mathcal{C}^{\infty}(\mathbb{T}^1,\mathbb{R}),$ $b$ either vanishes identically or it does not change sign, and $\lambda\doteq(\lambda_1,\ldots,\lambda_N)\in\mathbb{R}^N\setminus\{0\}.$ 

Given $\mu\in\mathcal{D}'(\mathbb{T}^{N+1})$ such that $L\mu=f\in\mathcal{E}_{\{\omega\}}(\mathbb{T}^{N+1}),$ by the partial Fourier series we are led to the equations
\[\partial_t\hat{u}(t,\xi)+i\langle\xi,\alpha(t)+ib(t)\lambda\rangle\hat{u}(t,\xi)=\hat{f}(t,\xi).\]

Since $(\alpha_0,b_0\lambda)$ satisfies $(EDC)_{1}^{\{\omega\}},$ it follows that for each $\xi$ we have either $\langle\xi,\alpha_0\rangle\not\in\mathbb{Z}$ or $b_0\langle\xi,\lambda\rangle\neq0.$ Hence each of the above equations has a unique solution which may be written in the two forms 

\[\hat{u}(t,\xi)=(1-e^{2\pi\langle\xi,b_0\lambda-i\alpha_0\rangle})^{-1}\int_{0}^{2\pi}\hat{f}(\xi,t-s)e^{\int_{t-s}^{t}\langle\xi,b(r)\lambda-i\alpha(r)\rangle dr}ds\]

or

\[\hat{u}(t,\xi)=(e^{-2\pi\langle\xi,b_0\lambda-i\alpha_0\rangle}-1)^{-1}\int_{0}^{2\pi}\hat{f}(\xi,t+s)e^{\int_{t}^{t+s}\langle\xi,-b(r)\lambda+i\alpha(r)\rangle dr}ds.\]

As explained in the proof of propositions \ref{pror6} and \ref{pror7}, the condition $(EDC)_{1}^{\{\omega\}}$ and the control on the change of sign of $b(t)$ will imply that the sequence of solutions $\hat{u}(t,\xi)$ has the decaying of a function $u$ in $\mathcal{E}_{\{\omega\}}(\mathbb{T}^{N+1}).$

Therefore, the conditions presented in Theorem \ref{mtr1} are sufficient to the global  $\{\omega\}-$hypoellipticity of $L.$

We now turn our attention to the necessity. Recall that in Lemma \ref{lemfinal} we showed that the global $\{\omega\}-$hypoellipticity of ${}^tL=-L$ implies that $L$ is strongly $\{\omega\}-$solvable.    

Since we already proved that conditions $(1)-(3)$ in Theorem \ref{mtr1} are necessary to the strong ${\{\omega\}}-$solvability, it follows that they are also necessary to the global  ${\{\omega\}}-$hypoellipticity. The proof of Theorem \ref{mtr1} is then completed.

{\subsection*{Acknowledgment}
The author gratefully acknowledge the suggestions of professor Bruno de Lessa Victor concerning ultradifferentiable functions.

\bibliographystyle{elsarticle-num}

\end{document}